\documentclass[
	fontsize=11pt,
	parskip=half,
]{scrartcl}

\usepackage[dvipsnames,hyperref]{xcolor}
			\definecolor{dark-gray}{gray}{0.1}
			\definecolor{dark-blue}{RGB}{0,0,80}
			\definecolor{light-gray}{gray}{0.9}

\usepackage[english]{babel}
\usepackage[autostyle]{csquotes}			
			\MakeOuterQuote{"}

\usepackage{amsmath} 					
\usepackage{mathtools}				
			\mathtoolsset{showonlyrefs,mathic}
\usepackage{amssymb}					
\usepackage{aliascnt}					
\usepackage{amsthm} 					
			\theoremstyle{definition}
				\newtheorem{definition}{Definition}[section]
				
				\newtheorem*{definition*}{Definition}
			\theoremstyle{plain}
				\newaliascnt{theorem}{definition} 								
				\newtheorem{theorem}[theorem]{Theorem}						
				\aliascntresetthe{theorem} 												
				\newtheorem*{theorem*}{Theorem}										
				\newaliascnt{proposition}{definition}
				\newtheorem{proposition}[proposition]{Proposition}
				\aliascntresetthe{proposition}
				
				\newaliascnt{corollary}{definition}
				\newtheorem{corollary}[corollary]{Corollary}
				\aliascntresetthe{corollary}
				
				\newtheorem*{corollary*}{Corollary}
				\newaliascnt{lemma}{definition}
				\newtheorem{lemma}[lemma]{Lemma}
				\aliascntresetthe{lemma}
				
				\newtheorem*{lemma*}{Lemma}
				\newtheorem{bigtheorem}{Theorem}

			\theoremstyle{remark}
				\newtheorem*{remark}{Remark}
				\newtheorem*{remarks}{Remarks}
				
				\newtheorem*{examples}{Examples}

\usepackage{array}						
\usepackage[inline]{enumitem}					

\usepackage[
	backend=biber,
	bibstyle=authoryear,
	citestyle=alphabetic,
	hyperref,
	sorting=ynt,
	giveninits=true,
	maxbibnames=99,
	backref=false,
	doi=false
	]{biblatex}
			\renewbibmacro{in:}{%
				\ifentrytype{article}{}{\printtext{\bibstring{in}\intitlepunct}}
			}
			\DefineBibliographyStrings{english}{backrefpage = {see p.}, backrefpages = {see pp.}}
			\renewbibmacro*{doi+eprint+url}{
				\iftoggle{bbx:doi}{\printfield{doi}}{}
				\iftoggle{bbx:eprint}{\usebibmacro{eprint}}{}
				\iftoggle{bbx:url}{
					\iffieldundef{doi}{
						\iffieldundef{eprint}{
							\usebibmacro{url+urldate}
						}{}
					}{}
				}
			}
			\AtEveryBibitem{\ifentrytype{article}{\clearfield{issn}}{}}
			\DeclareFieldFormat{labelalphawidth}{\mkbibbrackets{#1}}

			\defbibenvironment{bibliography}
				{\list
					 {\printtext[labelalphawidth]{%
							\printfield{labelprefix}%
					\printfield{labelalpha}%
							\printfield{extraalpha}}}
					 {\setlength{\labelwidth}{\labelalphawidth}%
						\setlength{\leftmargin}{\labelwidth}%
						\setlength{\labelsep}{\biblabelsep}%
						\addtolength{\leftmargin}{\labelsep}%
						\setlength{\itemsep}{\bibitemsep}%
						\setlength{\parsep}{\bibparsep}}%
						}
				{\endlist}
				{\item}

\usepackage[
	pdfusetitle,
	colorlinks=true, pdfborder={0 0 0},
	citecolor={dark-blue}, urlcolor={dark-blue},	linkcolor={dark-blue}
	]{hyperref}
\usepackage[numbered, open, openlevel=0]{bookmark}

\newcommand{\del}{\partial}
\DeclareMathOperator{\Tr}{Tr}
\DeclareMathOperator{\tr}{tr}

\newcommand{\after}{\circ}

\newcommand{\abs}[1]{\left| #1 \right|}
\newcommand{\norm}[1]{\left\lVert #1 \right\rVert}

\newcommand{\set}[1]{\left\{\ #1 \ \right\} } 							
\newcommand{\suchthat}{\mathrel{}\middle|\mathrel{}}		

\newcommand{\hodge}{\star}
\DeclareMathOperator{\vol}{vol}
\DeclareMathOperator{\Ric}{Ric}

\newcommand\Restrict[2]{{\left. \kern-\nulldelimiterspace #1 \right\vert_{#2} }}								 

\newcommand{\Z}{\mathbb{Z}} 		
\newcommand{\R}{\mathbb{R}} 		

\DeclareMathOperator{\ad}{ad}

\AtBeginDocument{%
\let\div\relax
\DeclareMathOperator{\div}{div}
}

\usepackage[left=2.5cm, right=2.5cm, top=3cm, bottom=3cm]{geometry}
\usepackage{setspace}
\let\savemathfrak\mathfrak											
\usepackage[ttscale=.85]{libertine} 
\usepackage[T1]{fontenc}
\usepackage[amsthm,upint]{libertinust1math} 
\let\mathfrak\savemathfrak											
\DeclareMathAlphabet\mathcal{OMS}{cmsy}{m}{n}		

\addtokomafont{paragraph}{\scshape\mdseries}
\addtokomafont{section}{\large\scshape}

\title{Growth of the Higgs Field for Kapustin-Witten solutions\\ on ALE and ALF gravitational instantons}

\begin{document}

\makeatletter
\renewcommand*{\thefootnote}{\fnsymbol{footnote}}

\thispagestyle{empty}

{

\setstretch{1.3}
\Large\bfseries\scshape
\@title

}

\vspace{1em}

Michael Bleher\footnotemark[2] \\
{\small\itshape
Institute for Mathematics, Heidelberg University, Im Neuenheimer Feld 205, Heidelberg, Germany.
}

\footnotetext[2]{mbleher@mathi.uni-heidelberg.de}

\renewcommand*{\thefootnote}{\arabic{footnote}}
\makeatother

\paragraph{Abstract.}
The $\theta$-Kapustin-Witten equations are a family of equations for a connection $A$ on a principal $G$-bundle $E \to W^4$ and a one-form $\phi$, called the Higgs field, with values in the adjoint bundle $\ad E$.
They give rise to second-order partial differential equations that can be studied more generally on Riemannian manifolds $W^n$ of dimension $n$.
For $G=SU(2)$, we report a dichotomy that is satisfied by solutions of the second-order equations on Ricci-flat ALX spaces with sectional curvature bounded from below.
This dichotomy was originally established by Taubes for $W^n=\R^n$; the alternatives are: either the asymptotic growth of the averaged norm of the Higgs field over geodesic spheres is larger than a positive power of the radius, or the commutator $[\phi\wedge\phi]$ vanishes everywhere.
As a consequence, we are able to confirm a conjecture by Nagy and Oliveira, namely, that finite energy solutions of the $\theta$-Kapustin-Witten equations on ALE and ALF gravitational instantons with $\theta\neq 0$ are such that $[\phi\wedge\phi]=0$, $\nabla^A \phi=0$, and $A$ is flat.

\section{Introduction}
Let $G=SU(2)$ and consider a principal $G$-bundle $E$ over a complete Riemannian manifold $(W^n, g)$ of dimenson $n$.
Throughout, we assume that $W^n$ is an ALX manifold.
Suffice it to say for now that we take this to mean $W^n$ is a non-compact manifold with fibered ends such that the $k$-dimensional fibers have bounded volume.
Consequently, the volume of geodesic balls asymptotically grows like $r^{n-k}$.

Denote by $(A,\phi) \in \mathcal{A}(E) \times \Omega^1(W^n, \ad E)$ a pair consisting of a connection on $E$ and an $\ad E$-valued one-form.
We write $\hodge$ for the Hodge star operator and equip $\Omega^k(W^n, \ad E)$ with the density-valued inner product ${\langle a,b \rangle = \Tr a \wedge\hodge b}$.
Upon integration this provides the usual $L^2$-product $\langle a, b\rangle_{L^2(W)} = \int_{W^n} \langle a , b \rangle$ on $\Omega^k(W^n, \ad E)$.
Throughout, we assume that $A$ and $\phi$ have enough derivatives and are locally square-integrable.

In this article we report on a property of the pair $(A,\phi)$ whenever it satisfies the following second order differential equation.
\begin{align} \label{eq:dichotomy-main-assumption}
	{\nabla^A}^\dagger \nabla^A \phi + \frac{1}{2} \hodge [ \hodge [\phi \wedge \phi] \wedge \phi] + \Ric \phi = 0 \ .
\end{align}
Here ${\nabla^A}^\dagger$ is the formal adjoint of $\nabla^A$ with respect to the $L^2$-product and the Ricci curvature is viewed as an endomorphism of $\Omega^1(W, \ad E)$.

The differential equation~\eqref{eq:dichotomy-main-assumption} is of particular relevance in the context of the Kapustin-Witten equations.
To see this, consider for the moment the case of a four-manifold $W^4$ and define the Laplace-type differential operator on $\Omega^1(W, \ad E)$:
\begin{align}
	\tilde \Delta_{A}(\phi) = - d_A d_A^\hodge \phi + \hodge 2 d_A ( d_A \phi)^- \ ,
\end{align}
where $d_A^\hodge= \hodge d_A \hodge$ is the usual codifferential and $(\cdot)^\pm$ denotes the (anti-)self-dual part of a given two-form.
Compare this operator with the $\theta$-Kapustin-Witten equations for $(A,\phi)$, which are given by
\begin{align}
	\big( \cos \tfrac{\theta}{2}\ ( F_A - \tfrac{1}{2} [\phi \wedge \phi] ) - \sin \tfrac{\theta}{2}\ d_A \phi\ \big)^+ &= 0 \\
	\big( \sin \tfrac{\theta}{2}\ ( F_A - \tfrac{1}{2} [\phi \wedge \phi] ) + \cos \tfrac{\theta}{2}\ d_A \phi\ \big)^- &= 0 \\
	d_A^{\hodge} \phi\ &= 0 \ .
\end{align}
Clearly, if $(A,\phi)$ is a solution of the $\theta = 0$ version of the Kapustin-Witten equations, then $\tilde \Delta_{A} \phi = 0$.
Moreover, using a Bochner-Weitzenböck identity that relates $\tilde \Delta_A$ and the Bochner Laplacian ${\nabla^A}^\dagger \nabla^A$, as well as the remaining part of the $0$-Kapustin-Witten equations $F_A^+ = [\phi\wedge\phi]^+$, one finds that harmonicity of $\phi$ with respect to $\tilde \Delta_A$ is equivalent to equation~\eqref{eq:dichotomy-main-assumption}.
In fact a very similar argument shows that the same is true if $(A,\phi)$ is a solution of the $\theta$-Kapustin-Witten equations \cite{Taubes2013, Taubes2017a, Nagy2021}.

The Kapustin-Witten equations arise from a family of topologically twisted, four-dimensional, supersymmetric gauge theories that exhibit surprisingly deep connections to several complementary areas of mathematics.
They were first studied by Kapustin and Witten in the context of the geometric Langlands program~\cite{Kapustin2007}.
Witten later realized that solutions of the Kapustin-Witten equations are also related to topological invariants of knots, specifically to Khovanov homology and its generalizations~\cite{Witten2011, Gaiotto2012a}.
Since then the moduli space of solutions to the Kapustin-Witten equations has been subject to extensive study \cite{Taubes2013, Taubes2017a, Taubes2018, He2015, He2019d, He2019a, Tanaka2019, Taubes2019, Taubes2021}.
\begin{remark}
Note that we use a slightly different normalization than is otherwise common in the literature: our $\theta$ coincides with $2\theta_{GU}$ in \cite{Gagliardo2012}.
Though this is mostly a matter of taste, when viewed as dimensional reductions of the five-dimensional Haydys-Witten equations, the normalization used here has a geometric interpretation as incidence angle between Haydys' preferred vector field and the hyperplane orthogonal to the reduction.
\end{remark}

Let us now return to general $n$-manifolds.
In what follows, we are guided by the intuition that if $\phi$ satisfies \eqref{eq:dichotomy-main-assumption}, then it is harmonic with respect to some well-behaved Laplace-type operator.
In particular, one should expect that it satisfies an appropriate analogue of the mean-value principle.
Hence, fix some point $p\in W^4$ and denote by $B_r$ the closed geodesic ball of radius $r$ centered at $p$.
Consider the non-negative function $\kappa$ defined on $[0,\infty)$ by
\begin{align}
	\kappa^2(r)
	= \frac{1}{r^{n-k-1}} \int_{\del B_r} \norm{\phi}^2 \ .
\end{align}
As a consequence of the asymptotic volume growth on $W^n$, $\kappa(r)$ is related to the average value of $\norm{\phi}$ on geodesic spheres $\del B_r$ with large radius $r$.
The mean-value principle for Laplace-type differential operators then suggests that $\phi$ should satisfy an inequality of the form $\norm{\phi(p)} \leq \kappa(r)$ for $r>0$.
Although contributions from non-trivial curvature in the interior of $B_r$ in general preclude this naive mean-value inequality, the controlled asymptotics of ALX spaces retains enough control to deduce analogous bounds for points that are far away from $p$.

A classical consequence of the mean-value principle is a relation between the asymptotic behaviour of $\kappa$ at large radius and the values of $\phi$ in the interior of the ball.
For example, if the naive mean-value inequality was satisfied at every point $p \in W^n$ and $\kappa(r) \to 0$ as $r \to \infty$, then $\phi$ would be identically zero everywhere.

For $W^n = \R^n$, a result by Taubes generalizes this kind of statement to a dichotomy between the growth of $\kappa(r)$ at infinity and the vanishing of $[\phi\wedge\phi]$ on all of $W^n$ \cite{Taubes2017a}.
Here we prove that this dichotomy holds more generally if $W^n$ is an ALX gravitational instanton (this is \autoref{thm:dichotomy} below):
\begin{bigtheorem} \label{bigthm:dichotomy}
Let $W^n$ be a complete, Ricci flat ALX manifold of dimension $n\geq 2$ with asymptotic fibers of dimension $k \leq n-1$ and sectional curvature bounded from below.
Consider $(A,\phi)$ as above and assume the pair satisfies the second-order differential equation~\eqref{eq:dichotomy-main-assumption}.
Then either
\begin{enumerate}
	\item there is an $a>0$ such that $\liminf_{r\to\infty} \frac{\kappa(r)}{r^a} > 0$, or
	\item $[\phi \wedge \phi]=0$.
\end{enumerate}
\end{bigtheorem}

If the fields $(A,\phi)$ are solutions of the $\theta$-Kapustin-Witten equations and have square-integrable field strength we can say slightly more (cf. \autoref{thm:dichotomy2}).
\begin{bigtheorem} \label{bigthm:dichotomy2}
Let $W^4$ be a complete, Ricci flat ALX manifold of dimension $4$ with asymptotic fibers of dimension $k \leq 3$ and sectional curvature bounded from below.
Assume $(A,\phi)$ are solutions of the $\theta$-Kapustin-Witten equations and satisfy $\int_{W^4} \norm{F_A}^2 < \infty$, then either
\begin{enumerate}
	\item there is an $a>0$ such that $\liminf_{r\to\infty} \frac{\kappa(r)}{r^a} > 0$, or
	\item $[\phi \wedge \phi]=0$, $\nabla^A \phi = 0$, and $A$ is self-dual if $\theta = 0$, flat if $\theta \in (0,\pi)$, and anti-self-dual if $\theta = \pi$.
\end{enumerate}
\end{bigtheorem}
As an immediate consequence of \autoref{bigthm:dichotomy2} we are able to confirm a conjecture of Nagy and Oliveira.
For this we introduce the Kapustin-Witten energy
\begin{align}
	E_{\text{KW}} = \int_{W^4} \left( \norm{F_A}^2 + \norm{\nabla^A \phi}^2 + \norm{[\phi\wedge\phi]}^2 \right) \ .
\end{align}
\begin{corollary*}[Nagy-Oliveira Conjecture \cite{Nagy2021}] \label{cor:nagy-oliveira-conjecture}
Let $(A,\phi)$ be a finite energy solution of the $\theta$-Kapustin-Witten equations with $\theta\neq 0 \pmod{\pi}$ on an ALE or ALF gravitational instanton and let $G=SU(2)$.
Then $A$ is flat, $\phi$ is $\nabla^A$-parallel, and $[\phi \wedge \phi] = 0$.
\end{corollary*}
\begin{proof}
Under the given assumptions, the main result of Nagy and Oliveira \cite[Main Theorem 1]{Nagy2021} states that $\phi$ has bounded norm and thus, in particular, bounded average over spheres.
It follows that $\liminf_{r\to\infty} \frac{\kappa(r)}{r^a} \to 0$ for any $a>0$, while the finite energy condition subsumes square-integrability of $F_A$.
Therefore, \autoref{bigthm:dichotomy2} implies that $[\phi\wedge\phi]=0$, $\nabla^A \phi=0$, and that $A$ is flat.
\end{proof}
\begin{remark}
The preceding argument is due to Nagy and Oliveira, who established the corollary for $W^4 = \mathbb{R}^4$ and $S^1 \times \mathbb{R}^3$ \cite[Corollary 1.3]{Nagy2021}.
Nagy and Oliveira relied on a version of \autoref{bigthm:dichotomy2} that applies to $W^4 = \mathbb{R}^4$ and was provided by Taubes alongside the original dichotomy \cite{Taubes2017a}.
Their conjecture stemmed from the expectation that Taubes' results can be extended to ALX spaces in general.
\end{remark}

The main insight of this article is that Taubes' proof strategy for the special case $W^n = \R^n$ carries over to general ALX spaces.
This is a consequence of the well-behaved asymptotic volume growth, where problems that arise from non-zero curvature in the interior can be excised.
Hence, the proof of \autoref{bigthm:dichotomy} closely follows the one provided by Taubes in \cite{Taubes2017a}.

We proceed as follows:
In \autoref{sec:dichotomy-setting} we collect the relevant definitions and recall several classical results that will be used throughout.
Then, in \autoref{sec:dichotomy-frequency-function}, we investigate the derivative of $\kappa$ and introduce the relevant analogue of Almgren's frequency function, as well as a function that captures contributions from the mean curvature of the geodesic sphere.
The key finding of that section is that $\kappa$ is asymptotically almost non-decreasing, which is a prerequisite for most of the heavy lifting in subsequent sections.
In \autoref{sec:dichotomy-unique-continuation}, we present a somewhat unusual version of unique continuation that is satisfied by $\kappa$.
The main insight is the content of \autoref{sec:dichotomy-slow-growth-bounded-frequency}, where we explain that slow asymptotic growth of $\kappa$ results in bounds for the frequency function.
All these results are refined with respect to the components of the one-form $\phi=\sum_i \phi_i dx^i$ by introducing in \autoref{sec:dichotomy-correlation-tensor} what we call the correlation tensor.
Using a second line of arguments, we also determine a priori bounds of the type $\norm{\phi(x)} \leq \kappa(r)$ in \autoref{sec:dichotomy-a-priori-bounds}, which are the anticipated analogues of the mean value inequalities mentioned already above.
Finally, all these ingredients are combined into a proof of \autoref{bigthm:dichotomy} in \autoref{sec:dichotomy-proof}, while the proof of \autoref{bigthm:dichotomy2} occupies \autoref{sec:dichotomy-proof2}.

\paragraph{Acknowledgements}
I thank Fabian Hahner for helpful comments on a draft of this article.
This work is funded by the Deutsche Forschungsgemeinschaft (DFG, German Research Foundation) under Germany's Excellence Strategy EXC 2181/1 - 390900948 (the Heidelberg STRUCTURES Excellence Cluster).

\section{Background}
\label{sec:dichotomy-setting}

For the purposes of this article, an ALX space is a non-compact complete Riemannian manifold that asymptotically looks like a fibration of closed manifolds, where the fibers have bounded volume.
This is made precise in the following definition.
\begin{definition}[ALX$_k$ Manifold]
Let $(W^n,g)$ be a complete Riemannian manifold of dimension $n$ and fix $p \in W^n$.
Let $\pi_Y: Y^{n-1} \to B^{n-k-1}$ be a fibration over an $n-k-1$-dimensional closed Riemannian base $(B, g_B)$ with $k$-dimensional closed Riemannian fibers $(X, g_X)$.
Equip $(0,\infty)_r \times Y$ with the model metric ${g_\infty = dr^2 + g_X + r^2 g_B}$.
We say $W^n$ is an ALX$_k$ manifold if its end is modeled on $(0,\infty) \times Y$, that is, if there exists $R > 0$ such that there is a diffeomorphism $\varphi: W^n \setminus B_{R}(p) \to (R, \infty) \times Y^{n-1}$ that satisfies for $j=0, 1$, and $2$
\begin{align}
	\lim_{r \to \infty} r^j \norm{ (\nabla^\text{LC})^j \left( g - \varphi^\ast g_\infty \right) }_{L^{\infty}(\del B_r)} = 0 \ .
	\label{eq:dichotomy-alx-metric}
\end{align}
\end{definition}

\begin{proposition} \label{prop:dichotomy-alx-volume-growth}
If $(W^n, g)$ is an ALX$_k$ space then
\begin{align}
	\vol B_r(p) \sim r^{n-k} \vol X  \quad  (r \to \infty) \ .
\end{align}
\end{proposition}

\begin{definition}
We call an ALX space $W^n$ a \emph{gravitational instanton} if it is Ricci flat and its sectional curvature is bounded from below.
\end{definition}

\begin{remarks}\leavevmode
\begin{itemize}
\item ALX spaces are usually considered in the context of four-manifolds and the "X" is a place-holder for the following cases:
ALE or Asymptotically Locally Euclidean ($k=0$), ALF or Asymptotically Locally Flat ($k=1$), ALG ($k=2$), ALH ($k=3$), where the last two are named by induction.

\item We do not demand that ALX gravitational instantons are hyper-Kähler, while we add the possibly non-standard condition of bounded sectional curvature.
\end{itemize}
\end{remarks}

\begin{examples}\leavevmode
\begin{itemize}
\item The prototypical example of an ALE manifold is Euclidean space $\R^n$. In this case there is the obvious diffeomorphism $\R^n \setminus \{0\} \simeq (0,\infty) \times S^{n-1}$ via spherical coordinates, the metric is $\varphi_\ast g = dr^2 + r^2 g_{S^{n-1}} = g_\infty$, and the volume of balls grows with the radius as $r^{n}$.
\item A prototypical example of an ALF space is $S^1 \times \R^{n-1}$ with the product metric $g = dt^2 + g_{\R^{n-1}}$.
Again, spherical coordinates on the $\R^{n-1}$ factor provide a diffeomorphism to $(0,\infty) \times S^1 \times S^{n-2}$ with metric $\varphi_\ast g = dr^2 + dt^2 + r^2 g_{S^{n-2}_r}$, such that $S^1$ has constant size, while $S^{n-2}_r$ is the sphere of radius $r$ centered at $0 \in \R^{n-1}$.
Once the volume of $S^1$ is "filled", the volume of a geodesic ball approaches a growth of order $r^{n-1}$.
\item Famous examples of less trivial four-dimensional ALF gravitational instantons are (multi-centered) Taub-NUT spaces.
These are $S^1$-fibrations over $\mathbb{R}^3$, where the $S^1$-fiber has asymptotically finite volume.
\end{itemize}
\end{examples}

\begin{theorem}[Global Laplacian Comparison Theorem \cite{Calabi1958}] \label{thm:dichotomy-laplacian-comparison}
If $\operatorname{Ric} \geq (n-1)K$ and $r(x) = d(p,x)$ denotes the geodesic distance function based at a point $p$, then
\begin{align}
	\Delta r \leq (n-1) \Delta_{K} r \ ,
\end{align}
where on the right hand side $\Delta_{K}$ is the Laplacian on the unique complete, $n$-dimensional, simply connected space of constant sectional curvature $K$.
\end{theorem}
\begin{remark}
Since $r$ is not necessarily differentiable the global Laplacian comparison must be understood in a weak sense, e.g. in the weak sense of barriers as in the work of Calabi \cite{Calabi1958}.
However, for our purposes it is sufficient to consider the smooth locus of $r$, where the inequality holds as stated.
\end{remark}

\begin{proposition}[Mean Curvature Comparison on ALX spaces] \label{prop:dichotomy-mean-curvature-comparison}
Let $W^n$ be an ALX${}_k$ space and fix a point $p\in W^n$.
The Laplacian of the distance function $r(x) = d(p,x)$, or equivalently the mean curvature of the geodesic sphere of radius $r$ based at $x$, has the following asymptotic behaviour.
\begin{align}
	\Delta r \sim \left\{\quad \begin{matrix} \frac{n-1}{r}  & (r \to 0)  \\[1em] \frac{n-k-1}{r} & (r \to \infty) \end{matrix} \right.
\end{align}
Furthermore, if $\Ric \geq 0$, then it is bounded from above by
\begin{align}
	\Delta r \leq \frac{n-1}{r} \ .
\end{align}
\end{proposition}
\begin{proof}
For a start, note that $r(x)$ is smooth on $M\setminus\{p, \mathrm{Cut}(p)\}$, where $\mathrm{Cut}(p)$ is the cut locus of $p$.
It is a standard result that the cut locus on a complete Riemannian manifold has measure zero, so $r$ is differentiable almost everywhere.
The Gauss lemma tells us that $\nabla^{\mathrm{LC}} r = \del_r$ is the radial vector field of unit norm and is normal to geodesic spheres.
As an aside, note that $\Delta r = \tr (\nabla^{\mathrm{LC}})^2 r$ is the trace of the second fundamental form of the geodesic sphere and as such is identical to its mean curvature.

The asymptotic behaviour for $r\to 0$ follows e.g. by a direct calculation in Riemann normal coordinates.
In particular, use $g_{ij} = \delta_{ij} + \mathcal{O}(r^2)$ and $\Gamma^{i}_{jk} = \mathcal{O}(r)$ and then observe that at leading order the result is identical to the Euclidean case, while higher order corrections are $\mathcal{O}(r)$:
\begin{align}
	\Delta r = \frac{n-1}{r} + \mathcal{O}(r) \ .
\end{align}
When $r \to \infty$ the ALX$_k$ condition \eqref{eq:dichotomy-alx-metric} implies that $\Delta r \sim \Delta_{\infty} r$, where $\Delta_\infty$ denotes the Laplacian associated to $g_\infty$ on $(0,\infty) \times Y^{n-1}$.
Under the diffeomorphism to $(0,\infty) \times Y^{n-1}$ the distance function is identified with the coordinate on the first factor.
Since the model metric is block diagonal and only depends on $r$ via the $r^2$ factor in front of $g_B$, we can calculate $\Delta_{g_\infty} r$ explicitly.
Let $(e_i)_{i=1,\ldots, n}$ be an orthonormal frame of $(0,\infty)\times Y$ such that $e_1 = \del_r$,  $e_2, \ldots, e_{k+1}$ are tangent to the fibers, and $e_{k+2}, \ldots, e_n$ are tangent to the base.
Write $\nabla$ for the Levi-Civita connection associated to $g_\infty$.
By a direct calculation $\nabla_{e_i} \del_r = \frac{1}{r} e_i$ for $i=k+2,\ldots, n-1$ and zero otherwise.
Hence,
\begin{align}
	\Delta_{g_\infty} r= \tr \nabla^2 r = g^{ij} g(\nabla_{e_i} \del_{r\ } , e_j) = \frac{\tr g_B}{r} = \frac{n-k-1}{r} \ .
\end{align}
Finally, the upper bound in the case that $\Ric$ is non-negative follows directly from the Laplacian Comparison Theorem (\autoref{thm:dichotomy-laplacian-comparison}).
\end{proof}

\begin{theorem}[Bishop-Gromov's Volume Comparison] \label{thm:dichotomy-bishop-gromov}
Let $(M,g)$ be a complete Riemannian manifold and assume $\operatorname{Ric} \geq (n-1) K$.
Denote by $\vol B_r(p)$ the volume of the geodesic ball of radius $r$ based at $p \in M$.
Similarly write $\vol_K B_r(p_K)$ for the volume of a geodesic ball with the same radius inside the unique complete, $n$-dimensional, simply connected space of constant sectional curvature $K$ at an arbitrary point $p_K$.
Then the function defined by
\begin{align}
	r \mapsto \frac{\vol B_r(p)}{\vol_K B_r(p_K)}
\end{align}
is non-decreasing and approaches $1$ as $r\to 0$.
In particular $\vol B_r(p) \leq \vol_K B_r(p_K)$ and  $\vol \del B_r(p) \leq \vol_K \del B_r(p_K)$.
\end{theorem}

\begin{lemma}\label{lem:dichotomy-bound-for-greens-function}
For any point $x$ in the interior of $B_r(p)$ there is a smooth, positive Green's function $G_x$ for the Dirichlet-Laplace problem on $B_r(p)$ with singularity at $x$, i.e. $\Delta G_x(y) = \delta_x(y)$ and $G_x(\del B_r(p)) = 0$.
If $W^n$ is a Ricci non-negative ALX${}_k$ space with effective dimension $n-k>2$, then for any $\epsilon>0$ there is a distance $D$ such that whenever $d(x,y)>D$ the Green's function is bounded by
\begin{align}
		G_x(y) \leq \frac{(1+\epsilon)\ c}{\vol X \: d(x,y)^{n-k-2}} \ ,
\end{align}
where the constant $c$ depends only on $n$.
\end{lemma}
\begin{proof}
The existence of a positive Green's function on compact, connected manifolds with boundary is standard.
The bound follows immediately from Theorem 5.2 in Li-Yau's seminal work \cite{Li1986}.
Their theorem states
\begin{align}
	G_x(y) \leq c \int_{r^2}^\infty \frac{1}{ \vol B_{\sqrt{t}} (x) } dt \ ,
\end{align}
where $r=d(x,y)$ denotes the Riemannian distance between $x$ and $y$ and the constant $c$ depends only on $n$.

Let $\epsilon > 0$. By \autoref{prop:dichotomy-alx-volume-growth} there is a distance $R\geq 0$, such that whenever $r\geq R$ we find
\begin{align}
	G_x(y) \leq c \int_{r^2}^\infty \frac{1+\epsilon}{\vol X\ t^{(n-k)/2}} dt
	= \frac{(1+\epsilon)\ c}{\vol X \: r^{n-k-2}} \ .
\end{align}
\end{proof}


\section{The Frequency Function}
\label{sec:dichotomy-frequency-function}

On our way to show that $\kappa$ must have some minimal asymptotic growth, the first step is to realize that its decay rate becomes arbitrarily small at large radii.
To see this we investigate the derivative of $\kappa$, which is given in the upcoming proposition.
The function $N(r)$ that arises in that context is an analogue of the frequency function as introduced by Almgren \cite{AlmgrenJr1979} and we will refer to it by that name.
The function $D(r)$ captures the average deviation of the mean curvature of the geodesic sphere from its limit at infinity.
\begin{proposition} \label{prop:dichotomy-derivative-of-kappa}
Assume the pair $(A,\phi)$ satisfies \eqref{eq:dichotomy-main-assumption}.
Whenever $\kappa$ is non-zero its derivative is
\begin{align} \label{eq:dichotomy-derivative-of-kappa}
	\frac{d\kappa}{dr} &=  \frac{\left( N + D \right) \kappa}{r} \ ,
\end{align}
where $N$ and $D$ are given by
\begin{align}
	N(r) &= \frac{1}{r^{n-k-2} \kappa^2} \int_{B_r} \left( \norm{ \nabla^A \phi }^2 + \norm{ [\phi \wedge \phi] }^2 + \langle \Ric \phi , \phi \rangle \right) \\[0.5em]
	D(r) &= \frac{1}{2 r^{n-k-2} \kappa^2} \int_{\del B_r} \left( \Delta r - \frac{n-k-1}{r} \right) \norm{\phi}^2 \ .
\end{align}
Moreover, if $\Ric \geq 0$, then $N$ is non-negative, $D$ is bounded from above by $k/2$, $\lim_{r\to 0} D = k/2$ and $\lim_{r\to\infty} D = 0$.
As a consequence, if $\kappa$ is not identically zero near $r=0$, then it is increasing on small enough neighbourhoods of $0$.
Similarly, if $\kappa$ is not asymptotically zero as $r\to \infty$, then it is asymptotically almost non-decreasing in the sense that for any $\epsilon>0$ there is some (large) radius $R$, such that $\frac{d\kappa}{dr} \geq - \frac{\epsilon \kappa}{r}$ for all $r \geq R$.
\end{proposition}
For notational convenience we will say that $\kappa$ is \emph{$\epsilon$-almost non-decreasing} whenever its derivative is bounded below by $\frac{d\kappa}{dr}\geq - \frac{\epsilon \kappa}{r}$.
\begin{proof}
Denote by $X$ the radial unit vector field on $B_r$ and observe that
\begin{align}
	\kappa^2(r)
	= \frac{1}{r^{n-k-1}} \int_{\del B_r} \norm{ \phi }^2
	= \frac{1}{r^{n-k-1}} \int_{B_r} \mathcal{L}_{X}\ \norm{ \phi }^2 \ .
\end{align}
By the product and Leibniz' integral rule the derivative is then given by
\begin{align}
	\frac{d}{dr} \kappa^2(r)
	&= - \frac{n-k-1}{r} \kappa^2 + \frac{1}{r^{n-k-1}} \int_{B_r} \mathcal{L}_X \circ \mathcal{L}_X\ \norm{ \phi }^2 \ .
\end{align}
We can write the integral on the right hand side equivalently as an integral over the trace of the (asymmetric) second Lie derivative ${\mathcal{L}^2_{Y,Z} := \mathcal{L}_{Y}\after \mathcal{L}_Z}$.
To see this denote by $(r,\theta_i)$ polar normal coordinates on $B_r$ and note that in these coordinates the metric is block-diagonal, i.e. $g = dr^2 + g_{S^{n-1}}$.
Since for any top-form $\omega$ the pullback of $\imath_{\del_{\theta_i}} \omega$ to the boundary of the geodesic ball is zero, one finds
\begin{align}
	\int_{B_r} \mathcal{L}^2_{X,X} \norm{\phi}^2
	= \int_{B_r} ( \mathcal{L}_{X,X}^2 + g_{S^{n-1}}^{ij} \mathcal{L}^2_{\del_{\theta_i},\del_{\theta_j}}) \norm{\phi}^2
	= \int_{B_r} \tr_{TM} \mathcal{L}^2 \norm{\phi}^2 \ .
\end{align}

Next, for any vector field $Y$ and top-form $\omega$ we may express the action of the Lie derivative in terms of the Levi-Civita connection as ${\mathcal{L}_Y \omega = \nabla_Y \omega + \div Y\ \omega}$.
Using this we may write the second Lie derivative as
\begin{align}
	\mathcal{L}_{Y,Z}^2\ \norm{\phi}^2
	&= ( \nabla_Y + \div Y )\ \nabla_Z \norm{\phi}^2 + \mathcal{L}_Y (\div Z\ \norm{\phi}^2) \ .
\end{align}
Furthermore, we use ad-invariance and metric compatibility to write $\nabla_Y \langle \phi, \phi \rangle = 2 \langle \phi, \nabla_Y^A \phi \rangle$, and use that the formal adjoint is given by ${\nabla_Y^A + \div Y = - (\nabla_Y^A)^\dagger}$.
This leads to
\begin{align}
	\int_{B_r} \tr_{TM} \mathcal{L}^2 \norm{\phi}^2
	&= 2 \int_{B_r} \left( \norm{ \nabla^A \phi }^2 - \langle \phi, {\nabla^A}^\dagger \nabla^A \phi \rangle \right) + \int_{B_r}  \tr_{TM}\left( \mathcal{L}_{\cdot} \div(\cdot) \ \norm{\phi}^2 \right) \\
	&= 2 \int_{B_r} \left( \norm{ \nabla^A \phi }^2 + \norm{ [\phi \wedge \phi] }^2 + \langle \phi, \Ric \phi \rangle \right) + \int_{\del B_r} \Delta r\ \norm{\phi}^2 \ ,
\end{align}
where we used the second order differential equation \eqref{eq:dichotomy-main-assumption} in the first term and that the only non-zero contribution in the second term contains the mean curvature of the geodesic sphere since $\div X = \Delta r$.

All in all, as long as $\kappa \neq 0$, the derivative is given by
\begin{align}
	\frac{d\kappa}{dr}
	&= \frac{1}{2 \kappa} \frac{d}{dr} \kappa^2
	= \frac{1}{\kappa r^{n-k-1}} \int_{B_r} \left( \norm{ \nabla^A \phi }^2 + \norm{ [\phi \wedge \phi] }^2 + \langle \phi, \Ric \phi \rangle \right)
	+ \frac{1}{2 \kappa r^{n-k-1}} \int_{\del B_r} \left( \Delta r - \tfrac{n-k-1}{r} \right) \norm{\phi}^2 \ ,
\end{align}
which upon identifying the terms on the right hand side with $N$ and $D$ is the desired result.

Now assume $\Ric \geq 0$.
On the one hand, $N$ is then clearly non-negative.
On the other hand, the results for $\Delta r$ from \autoref{prop:dichotomy-mean-curvature-comparison} immediately provide both the global upper bound and the limits of $D$.

Combining these facts with the formula for $\frac{d\kappa}{dr}$ leads to the conclusion that $\kappa$ is (almost) non-decreasing at both ends:
Since $D$ is continuous and $\lim_{r\to 0} D = k/2$, $D$ must be positive on some small interval $[0,s)$.
Thus, if $\kappa$ is non-zero somewhere in that interval then it is increasing.
The asymptotic bound for $r\to \infty$ works out similarly.
In that case there is an interval $[R,\infty)$ for any $\delta>0$ on which
\begin{align}
	D \geq - \frac{\delta}{1+\delta} \frac{n-k-1}{2} \ .
\end{align}
After a suitable choice of $\delta$ this provides the desired bound $\frac{d\kappa}{dr}\geq - \frac{\epsilon \kappa}{r}$ for any $\epsilon>0$, which concludes the proof.
\end{proof}

In the preceding proposition we already encountered lower bounds for $\frac{d\kappa}{dr}$ near $r=0$ and $r \to \infty$.
But when we keep track of $N$ it becomes clear that $\frac{d\kappa}{dr}$ satisfies stronger bounds than recorded so far.
This is the content of the following two corollaries.
The first records a global growth limitation, while the second determines asymptotic lower and upper bounds, both in dependence of the frequency function $N$.
\begin{corollary}\label{cor:dichotomy-kappa-global-grönwall}
Assume $\kappa\neq 0$ on $[r_0,r_1]$, then
\begin{align}
	\kappa(r_1)\ \leq \ \kappa(r_0)\ \exp \int_{r_0}^{r_1} \frac{N(t) + k/2}{t} dt\ .
\end{align}
\end{corollary}
\begin{proof}
Recall from \autoref{prop:dichotomy-derivative-of-kappa} that $D\leq k/2$, so the derivative of $\kappa$ is bounded by ${\frac{d\kappa}{dr} \leq \frac{(N+k/2)\kappa}{r}}$.
By Grönwall's inequality $\kappa$ then can't become larger than a solution of the underlying differential equation, which is the stated bound.
\end{proof}

\begin{corollary} \label{cor:dichotomy-kappa-asymptotic-grönwall}
Let $\epsilon > 0$ and $R$ be such that $\abs{D} \leq \epsilon$ on $[R,\infty)$.
If $\kappa$ has no zeroes in $[r_0,r_1]\subset[R,\infty)$ then it is bounded at $r_1$ from both sides as follows
\begin{align} \label{eq:dichotomy-kappa-asymptotic-grönwall}
	\kappa(r_0)\ \exp \int_{r_0}^{r_1} \frac{N(t)-\epsilon}{t} dt\
	\leq\ \kappa(r_1)\
	\leq\ \kappa(r_0)\ \exp \int_{r_0}^{r_1} \frac{N(t)+\epsilon}{t} dt\ .
\end{align}
Consequently, bounds for the frequency function on $[r_0,r_1]$ have the following effect:
\begin{itemize}
	\item if $a\leq N$ then $\kappa(r_1) \geq \big( \frac{r_1}{r_0} \big)^{a-\epsilon} \kappa(r_0)$
	\item if $N \leq b$ then $\kappa(r_1) \leq \big(\frac{r_1}{r_0} \big)^{b+\epsilon} \kappa(r_0)$
\end{itemize}
\end{corollary}
\begin{proof}
Since $\abs{D}\leq \epsilon$ and $\kappa$ is non-zero on $[r_0,r_1]$, its derivative is bounded in both directions as follows
\begin{align}
	\frac{(N-\epsilon)\ \kappa}{r} \leq \frac{d \kappa}{dr} \leq \frac{(N+\epsilon)\ \kappa}{r} \ .
\end{align}
Grönwall's inequality states that $\kappa$ is then bounded in either direction by the corresponding solutions of the underlying differential equations, which is exactly \eqref{eq:dichotomy-kappa-asymptotic-grönwall}.

Clearly, if the frequency function is bounded below by some $a\geq 0$ the first inequality in \eqref{eq:dichotomy-kappa-asymptotic-grönwall} reduces to
\begin{align}
	\kappa(r_1) \geq \kappa(r_0)\ \exp \int_{r_0}^{r_1} \frac{a-\epsilon}{s}\; ds
	= \left( \frac{r_1}{r_0} \right)^{a-\epsilon} \kappa(r_0)\ .
\end{align}
The same argument, but based on the second inequality, yields the corresponding upper bound if $N$ is bounded from above.
\end{proof}

We will later also need the derivative of $N$, which is given directly by the product and Leibniz' integral rule.
\begin{align}
	\frac{d}{dr} N
	&= \frac{1}{r^{n-k-2} \kappa^2}\ \int_{\del B_r} \bigg( \lVert \nabla^A \phi \rVert^2 + \norm{ [\phi \wedge \phi] }^2 + \langle \Ric \phi, \phi \rangle \bigg)
	- \big( n-k-2 + 2 (N + D) \big) \frac{N}{r}
	\label{eq:dichotomy-derivative-of-N}
\end{align}
As an immediate consequence, we see that if $N$ is ever small, then it can't have been very much larger at nearby smaller radii $s < r$.
This observation is recorded more precisely in the following proposition.
\begin{proposition}\label{prop:dichotomy-small-N}
Assume $\Ric\geq 0$.
If $N \leq 1$ on some interval $[r_0, r_1]$, then $N(r_0) \leq \Bigl( \frac{r_1}{r_0} \Bigr)^{n} N(r_1)$.
Moreover, whenever $N(r) < 1$ at some $r \in (0, \infty)$ then $N \leq 1$ on the interval $[ N(r)^{1/n}\: r , r]$.
\end{proposition}
\begin{proof}
Since $\Ric$ is non-negative the same is true for the first term in \eqref{eq:dichotomy-derivative-of-N}.
Moreover, in that case $D\leq \frac{k}{2}$ by \autoref{prop:dichotomy-derivative-of-kappa}.
Assume now that $N \leq 1$ on all of $[r_0,r_1]$.
Then $N$ satisfies the following differential inequality for any $r\in [r_0,r_1]$
\begin{align}
	\frac{dN}{dr} \geq - \frac{n\: N}{r}\ .
\end{align}
Grönwall's inequality states that then for any pair $s\leq r$ in $[r_0, r_1]$ the following inequality holds
\begin{align}\label{eq:dichotomy-frequency-function-grönwall}
	N(r) \geq \left( \frac{s}{r} \right)^{n} N(s)\ ,
\end{align}
which proves the first part of the statement.

Now assume $N(r) < 1$ for some $r \in (0,\infty)$.
By continuity $r$ must be contained in some interval $[r_0,r_1]$ on which $N\leq 1$, so we can use the preceding inequality in the form $N(s) \leq (r/s)^{n} N(r)$.
The right hand side is less than or equal to $1$ as long as $s \geq  N(r)^{1/n}\: r$.
\end{proof}


\section{Asymptotically Unique Continuation}
\label{sec:dichotomy-unique-continuation}

Any function that is both non-negative and non-decreasing (trivially) has the following property: if it is non-zero at a particular point $r_0$, it will remain non-zero for any subsequent point $r > r_0$.
Although $\kappa$ is not non-decreasing, there is some (large) radius $R$ beyond which it behaves in that same way.
This is a consequence of the fact that the decay rate of $\kappa$ becomes arbitrarily small at infinity.

To make this precise, recall that $\kappa$ is continuous, non-negative, and $\epsilon$-almost non-decreasing at large radius.
With respect to the last property fix some $\epsilon>0$ with associated radius $R\geq 0$.
Assume there is an $r_1 \in [R,\infty)$ at which $\kappa(r_1) \neq 0$.
This is the case, for example, if $\kappa$ is not asymptotically equivalent to the zero function.
Since $\kappa$ is $\epsilon$-almost non-decreasing, \autoref{cor:dichotomy-kappa-asymptotic-grönwall} provides a strictly positive lower bound for any larger radius $r\geq r_1$
\begin{align}
	\kappa(r) \geq \left( \frac{r_1}{r} \right)^{\epsilon} \kappa(r_1)\ ,
\end{align}
which prevents $\kappa$ from vanishing at any larger radius, at least as along as $r_1\neq 0$.
Note that, if $D$ is bounded from below, then there is a choice of $\epsilon$ for which $R=0$ such that the conclusion holds for any $r_1 \in (0,\infty)$.
In any case, the set on which $\kappa$ is non-zero must include an interval $(r_0, \infty)$ with some (possibly infinite) $r_0\geq 0$.

Let us now investigate the behaviour near $r_0$, where the growth rate of $\kappa$ is controlled by the frequency function $N$ and the mean curvature deviation $D$.
Observe that if $r_0\neq 0$ and $N$ is bounded from above on $(r_0, r_1]$, then the assumptions $\kappa(r_0)=0$ and $\kappa(r_1)\neq 0$ lead to a contradiction.
To see this use \autoref{cor:dichotomy-kappa-global-grönwall} with the upper bound $N\leq b$, which yields
\begin{align}
	\kappa(r) \geq \left( \frac{r}{r_1} \right)^{b+k/2} \kappa(r_1) \quad \text{for all } r \in (r_0, r_1]\ .
\end{align}
The right hand side is strictly positive and thus prevents $\kappa$ from vanishing at $r_0$.

This is an instance of Aronszajn's unique continuation theorem, which states that a non-trivial function that satisfies a second order, elliptic differential inequality cannot exhibit zeroes of infinite order \cite{Aronszajn1957}.
A particular consequence of this is that a non-negative and non-decreasing such function on $(0,\infty)$ is then either identically zero or strictly positive across the entire domain.

It was noted by Taubes that $\kappa$ -- in fact its defining integral $\int_{\del B_r} \norm{\phi}^2$ -- satisfies Aronszajn's unique continuation theorem \cite[Sec. 3]{Taubes2017}.
We obtain the following, slightly weaker version of the preceding statement:
\begin{lemma} \label{lem:dichotomy-unique-continuation}
Assume $\operatorname{Ric}\geq 0$ and that sectional curvature is bounded from below.
There is a radius $R\geq 0$, such that if $\kappa$ is non-zero at any point in $[R,\infty)$, then it is strictly positive on all of $(0,\infty)$.
Moreover, $R=0$ if the mean curvature deviation $D$ is bounded from below.
\end{lemma}
Here it is possible that $\kappa$ is compactly supported, but only as long as its support is restricted to a region where the mean curvature of the geodesic spheres is (still) large.
In view of the discussion above, the Lemma also follows directly from a proof that $N$ is a priori bounded on any interval $(r_0,r_1]$ in which $\kappa$ does not have zeroes.

\section{Slow Growth and  Bounded Frequency}
\label{sec:dichotomy-slow-growth-bounded-frequency}

We have previously seen that an upper bound for $N$ leads to bounded growth of $\kappa$.
The goal of this section is to show that the converse is true when $r\to\infty$.
More precisely, we show that whenever $\kappa$ grows slower than $\mathcal{O}(r^a)$ between two large radii, then $N$ must have been bounded from above on an interval leading up to the violation.
Note that such violations must occur for arbitrarily large radii when $\kappa$ is not asymptotically bounded below by $r^a$, which is the situation of the second alternative in \autoref{bigthm:dichotomy}.
Accordingly, the upcoming lemma and its refinement in \autoref{sec:dichotomy-correlation-tensor} play a crucial role in the proof of the main theorem.

\begin{lemma}\label{lem:dichotomy-frequency-function-bounds}
Assume $\kappa$ is not asymptotically zero.
Fix an $\epsilon > 0$ and denote by $R$ the radius beyond which $\abs{D}\leq \epsilon$.
If there is a pair of radii $r_0 \leq r_1$ in $[R,\infty)$ for which $\kappa(r_1) \leq \left(\frac{r_1}{r_0}\right)^{a-\epsilon} \kappa(r_0)$, then there exists a radius $t \in [r_0, r_1]$ such that
\begin{enumerate}
	\item $N(t) \leq a$.
\end{enumerate}
Moreover, if $a<1$ the following holds on the interval $[\tilde R, t]$, where ${\tilde R=\max\bigl( a^{\frac{1}{2n}}\: t,\; R\bigr)}$.
\begin{enumerate}[resume]
	\item $N < \sqrt{a}$,
	\item $\kappa \geq a^{\frac{\sqrt{a}+\epsilon}{2n}} \kappa(t)$.
\end{enumerate}
\end{lemma}

\begin{proof}
To see \textit{(i)} assume to the contrary that $N > a $ on all of $[r_0, r_1]$.
Then the first bullet in \autoref{cor:dichotomy-kappa-asymptotic-grönwall} states that
\begin{align}
	\kappa(r_1) > \left(\frac{r_1}{r_0}\right)^{a-\epsilon} \kappa(r_0)\ ,
\end{align}
which violates the assumption that $\kappa(r_1)$ satisfies exactly the opposite inequality.
Hence, there is a $t \in [r_0,r_1]$ at which $N(t) \leq a$.

Assume now that $a < 1$ and note that then the same is true for $N(t)$.
In that case we conclude via the second part of \autoref{prop:dichotomy-small-N} that $N\leq 1$ on $[N(r)^{1/n}\: t , t]$.
Since $N(t)\leq a < \sqrt{a}$, this interval contains as a subinterval $[a^{1/2n}\: t , t]$.
Then the first part of \autoref{prop:dichotomy-small-N} for ${s \in [a^{1/2n}\: t , t]}$ yields
\begin{align}
	N(s) \leq \left(\frac{r}{s}\right)^{n-k} N(t)
	\leq \frac{1}{\sqrt{a}} N(t)
	\leq \sqrt{a}\ ,
\end{align}
which proves that the same is true on the (possibly smaller) interval $[\tilde R , t]$ with ${\tilde R:=\max\bigl( a^{1/2n}\: t, R\bigr)}$.

Since $N\leq \sqrt{a}$ on $[\tilde R , t ]$ the second bullet of \autoref{cor:dichotomy-kappa-asymptotic-grönwall} provides the bound
\begin{align}
	\kappa(t)
	\leq \left(\frac{r}{t}\right)^{\sqrt{a}+\epsilon} \kappa(r)
	\leq a^{-\frac{\sqrt{a}+\epsilon}{2n}} \kappa(r)\ ,
\end{align}
where in the last step we used $a^{1/2n}\: t \leq \tilde R \leq r$ for any $r\in [\tilde R, t]$.
\end{proof}


\section{The Correlation Tensor}
\label{sec:dichotomy-correlation-tensor}

There is an $\Omega_p^1 \otimes \Omega_p^1$-valued function $T$ that refines $\kappa^2$ to the effect that it resolves the behaviour of the components of $\phi$.
Note that, being a one-form, $\phi$ can be evaluated in particular on the covariantly constant unit vector field on $B_r(p)$ that is defined by parallel transport of a unit vector $v \in T_p W$ along radial geodesics.
The output is a smooth function $\phi_v$ on $B_r(p)\setminus \operatorname{Cut}(p)$ that captures the evolution of the $v$-component of $\phi$ along the geodesics emanating from $p$.
This allows the definition of what we will call the correlation tensor $T: W^4 \times (0,\infty) \to \Omega_p^1 \otimes \Omega_p^1$.
Its value for $v,w \in T_p W$ is defined by
\begin{align}
	T(p,r)(v, w) &= \frac{1}{r^{n-k-1}} \int_{\del B_r(p)} \langle \phi_v, \phi_w \rangle\ .
	\label{eq:dichotomy-correlation-tensor}
\end{align}
Note, in particular, that $\tr_{T_p W} T(p,r) = \kappa^2(p,r)$, while the induced quadratic form $\kappa_v^2 := T(v,v)$ returns a version of the function $\kappa^2$ that is based on the component $\phi_v$.
Just like $\kappa_v(p,r)$ has an interpretation as the mean value of $\phi_v$ on the geodesic sphere, the value of $T(v,w)$ measures the correlation between the components $\phi_v$ and $\phi_w$ on the geodesic sphere; hence its name.

An important observation is that $\kappa_v$ satisfies essentially the same properties as $\kappa$.
As a short motivation of that fact observe that if $\tilde \Delta_A \phi = 0$ and $v$ denotes the covariantly constant vector field described above, then also ${\imath_v \tilde \Delta_A \phi = \tilde \Delta_A \phi_v = 0}$.
So it is reasonable to expect that $\kappa_v$ behaves like a harmonic function in exactly the same way that $\kappa$ does.

To make this more precise, if $\phi$ satisfies our main assumption \eqref{eq:dichotomy-main-assumption} then $\phi_v$ satisfies the following analogous second order equation
\begin{align}
	{\nabla^A}^\dagger \nabla^A \phi_v + \frac{1}{2} \hodge [ \hodge [\phi \wedge \phi_v] \wedge \phi] + \Ric(\phi)(v) = 0\ .
	\label{eq:dichotomy-main-assumption-weighted}
\end{align}
As a consequence the derivative of $\kappa_v$ comes with its own version of the frequency function and mean curvature deviation, denoted $N_v$ and $D_v$.
\begin{align}
	\frac{d\kappa_v}{dr} = \frac{N_v + D_v}{r} \kappa_v
	\label{eq:dichotomy-derivative-of-weighted-kappa}
\end{align}
The functions $N_v$ and $D_V$ are given by essentially the same expressions as before, but with $\phi$ replaced by $\phi_v$ as follows.
\begin{align} \label{eq:dichotomy-weighted-frequency-function}
	N_v &= \frac{1}{r^{n-k-2} \kappa_v^2} \int_{B_r} \left( \lVert \nabla^A \phi_v \rVert^2 + \norm{[\phi\wedge\phi_v]}^2 + \left\langle \Ric(\phi)(v), \phi_v \right\rangle \right) \\
	D_v &= \frac{1}{r^{n-k-1} \kappa_v^2} \int_{\del B_r} \left( \Delta r - \frac{n-k-1}{r} \right) \norm{\phi_v}^2
\end{align}
\begin{proposition}
Let $v \in T_pW$. All previous results hold verbatim when we replace $\kappa$ and $N$ by $\kappa_v$ and $N_v$, respectively.
\end{proposition}

There are in fact analogous results for the correlation tensor $T$.
To see this define the (Frobenius) norm of $T$ with respect to the inner product on $\Omega_p^1\otimes\Omega_p^1$ induced by the metric, i.e.
\begin{align}
	\norm{T}^2 = g^{\mu\rho} g^{\nu\sigma} T_{\mu\nu} T_{\rho\sigma}\ .
\end{align}
In this expression the metric is evaluated at the point $p$.
The norm of $T$ satisfies $\frac{1}{c} \kappa^2 \leq \norm{T} \leq \kappa^2$ for some constant $c$.
The tensor $T$ is differentiable with respect to $r$ and there is then a (possibly larger) $c$ such that the following inequality holds.
\begin{align}
	\norm{ \frac{dT}{dr} } \leq c\ \frac{N+D}{r} \norm{T}
\end{align}
Below we will also use the notation $T^\prime := \frac{dT}{dr}$.

From now on view $T$ as linear map from $T_p W \to T_p W$.
If $T$ has a zero eigenvalue at some radius $r$ and $v$ denotes the associated eigenvector, then $\kappa_v^2(r) = 0$.
As a consequence of the unique continuation property in \autoref{lem:dichotomy-unique-continuation}, whenever $r$ is too large $\kappa_v$ must then be identically zero on an interval of the form $[R,\infty)$.
This in turn implies that the component $\phi_v$ vanishes on $W^n \setminus B_R$.
We will deal with such compactly supported components of $\phi$ later.
In the definition of $T$ we restrict ourselves to the subspace of $T_p W$ that is orthogonal to the zero eigenspace of $T$ at infinity.
The correlation tensor then has strictly positive eigenvalues on all of $(0,\infty)$, which will be assumed henceforth.

Denote by $\lambda:(0,\infty) \to \R$ the function that assigns to a radius $r$ the smallest eigenvalue of $T(r)$.
By the preceding paragraph we may assume that this function is non-zero on all of $(0,\infty)$.
Clearly $\lambda \leq \kappa^2$ as the latter is the trace of $T$, so it should be noted that under this assumption $\kappa$ is necessarily everywhere non-zero.

The following proposition states that $\lambda$ is nearly differentiable.
\begin{proposition}\label{prop:dichotomy-smallest-eigenvalue-function}
The finite differences of the smallest-eigenvalue function $\lambda$ at any $r\in (0,\infty)$ satisfy the following bounds:
Denote by $v$ a unit length eigenvector of $T(r)$ with eigenvalue $\lambda(r)$ and let $\Delta > 0$.
\begin{align}
	\lambda(r+\Delta) - \lambda(r) &\leq \langle v , T^\prime(r)\; v \rangle\ \Delta + \mathcal{O}(\Delta^2)\\[0.4em]
	\lambda(r) - \lambda(r-\Delta) &\geq \langle v , T^\prime(r)\; v \rangle\ \Delta + \mathcal{O}(\Delta^2)
\end{align}
Moreover, $\lambda$ is locally Lipschitz continuous on $(0,\infty)$.
\end{proposition}
\begin{proof}
For any $r\in(0,\infty)$ denote by $v_r$ an eigenvector of $T(r)$ with eigenvalue $\lambda(r)$.
For $\Delta\geq 0$ we then have
\begin{align}
	\lambda(r+\Delta) - \lambda(r)
	= \langle v_{r+\Delta} , T(r+\Delta) v_{r+\Delta} \rangle - \langle v_r , T(r) v_r \rangle
	\leq \langle v_r , \big( T(r+\Delta)-T(r) \big) v_r \rangle\ .
\end{align}
To get the upper bound on the right hand side note that the smallest eigenvalue at $r+\Delta$ is $\lambda(r+\Delta)$, so the evaluation of $T(r+\Delta)$ on $v_r$ can only ever result in the same or larger values.
The upper bound for the forwards finite difference is then a consequence of Taylor's theorem.
Up to a minus sign the lower bound for the backwards finite difference in the second inequality follows in exactly the same way.

The Lipschitz property of $\lambda$ follows by paying closer attention to the remainder in Taylor's theorem.
Consider a pair $s<r \in (0,\infty)$ and Taylor's theorem at zeroth order
\begin{align}
	T(r) = T(s) + R_0(r)\ .
\end{align}
A standard estimate for the remainder states $\norm{R_0(r)} \leq \sup_{(r,s)} \norm{T^\prime}\ (r-s)$, as long as the derivative is bounded on the given interval.
As mentioned before, the derivative of $T$ at any radius $\tilde r$ is bounded by a multiple of $\frac{N+D}{\tilde r} \norm{T(\tilde r)}$.
Recall from \autoref{prop:dichotomy-derivative-of-kappa} that $D\leq \frac{k}{2}$.
Also, as was discussed in the context of the unique continuation property, $N$ is bounded on any compact interval on which $\kappa$ is non-zero.
It follows that $\lambda$ is Lipschitz on any compact $[s,t]\subset (0,\infty)$, since
\begin{align}
	\abs{\lambda(r) - \lambda(s)}
	\leq \abs{ \langle v_s , \big( T(r)-T(s) \big) v_s \rangle }
	\leq c \abs{r-s}\ .
\end{align}
Equivalently (since $(0,\infty)$ is locally compact): $T$ is locally Lipschitz.
\end{proof}

We can use the preceding proposition to show that, similar to $\kappa$, the function $\lambda$ is asymptotically almost non-decreasing.
The idea here is that $\lambda$ consists of piecewise smooth segments, each coinciding with a function $\kappa_v^2$ for some $v$.
Any $\kappa_v$ is eventually $\epsilon$-almost non-decreasing, so an analogues statement must be true for $\lambda$.
Observe that asymptotic bounds for the mean curvature deviation $D_v$ only depend on the asymptotics of the mean curvature $\Delta r$, which is independent of $v$.
Thus, if $D$ is bounded from below by $-\epsilon$, then the same is true for $D_v$.

Hence, let $R$ be such that $\abs{D}\leq \epsilon$ for any larger radius and consider points $r>s \geq R$.
Define an equidistant partition of the interval $[s,r] = \bigcup_{k=0}^{M-1} [r_k, r_{k+1}]$ where $r_k = s + \frac{k}{M} (r-s)$.
Then, denoting by $v_k$ an eigenvector with eigenvalue $\lambda(r_k)$, the two estimates in \autoref{prop:dichotomy-smallest-eigenvalue-function} imply
\begin{align}
	\lambda(r) = \lambda(s) + \sum_{k=1}^M \langle v_k , T^\prime(r_k)\; v_k \rangle \frac{r-s}{M} + \mathcal{O}\left(\frac{1}{M^2} \right)\ .
\end{align}
This can be rewritten by using $\langle v , T' v \rangle = \frac{d}{dr} \kappa_{v}^2$ and the fact that each $v_k$ is an eigenvector with the smallest eigenvalue at $r_k$, i.e. $\kappa_{v_k}^2(r_k) = \lambda(r_k)$.
\begin{align}
	\lambda(r) = \lambda(s) + 2 \sum_{k=1}^M \lambda(r_k) \frac{\left( N_{v_k}(r_k)\ + D_{v_k}(r_k)\right)}{r_k}\ \frac{r-s}{M} + \mathcal{O}\left( \frac{1}{M^2} \right)
	\label{eq:dichotomy-smallest-eigenvalue-is-almost-nondecreasing}
\end{align}

Notably, since $\abs{D}\leq \epsilon$ we find that $\lambda$ satisfies the following inquality on $[R,\infty)$
\begin{align}
	\frac{\lambda(r)-\lambda(s)}{r-s} &\geq - \frac{2 \epsilon}{M} \sum_{k=1}^M \frac{\lambda(r_k)}{r_k}  + \mathcal{O}\left( \frac{1}{M^2} \right)\ .
\end{align}
This is the finite difference analogue of the differential inequality satisfied by $\kappa$ when it is $\epsilon$-almost non-decreasing.
To make this more apparent, observe that the right hand side of the last inequality contains the arithmetic mean of the function $\lambda/r$ for the given partition of $[s,r]$.
If one takes $M\to\infty$ this expression approaches the mean value of $\lambda/r$ on the given interval.
Seeing that $\lambda$ is continuous, it is then a consequence of the (integral) mean value theorem that there is a point $t \in [s,r]$ at which the right hand side is given by $\lambda(t)/t$, such that
\begin{align}
	\frac{\lambda(r) - \lambda(s)}{r-s} \geq - 2 \epsilon \frac{\lambda(t)}{t}\ .
\end{align}
If $\lambda$ is differentiable at $r$ the limit $s \to r$ exactly recovers the differential inequality, as promised.
We will correspondingly say that $\lambda$ is $\epsilon$-almost non-decreasing whenever it satisfies the finite difference inequality from above.

The fact that $\lambda$ is eventually $\epsilon$-almost non-decreasing allows us to extend \autoref{lem:dichotomy-frequency-function-bounds} in such a way that it also provides bounds for $N_v$ and $\kappa_v$, where $v$ is a unit eigenvector of $T(t)$ associated to the smallest eigenvalue $\lambda(t)$ at some distinguished $t \in [r_0,r_1]$.
\begin{lemma}\label{lem:dichotomy-weighted-frequency-function-bounds}
Fix $\epsilon > 0$ and denote by $R$ the radius beyond which $\abs{D}\leq \epsilon$.
Let $r_0, r_1 \in [R, \infty)$ be a pair of radii such that $r_1$ is larger than any of $4 r_0$, $\left( \kappa^2(r_0)/\lambda(r_0) \right)^{\frac{1}{2a}} r_0$, and $r_0^{1/({1-100\sqrt{a}} )}$.
If ${\kappa(r_1) \leq \left(\frac{r_1}{r_0}\right)^{a-\epsilon} \kappa(r_0)}$ and $a$ is sufficiently small (e.g. $a^{1/4n}<0.1$), then there exists a radius $t \in [r_1^{1-100\sqrt{a}}, r_1]$ of the following significance:

Let $v$ be an eigenvector of $T(t)$ associated to the smallest eigenvalue $\lambda(t)$ and write ${\tilde R = \max\left( a^{\frac{1}{8n}} t , R \right)}$.
On all of $[\tilde R, t]$
\begin{enumerate}
	\item $N \leq a^{1/4}$ and $\kappa \geq a^{\frac{a^{1/4} + \epsilon}{4n}} \kappa(t)$
	\item $N_v \leq a^{1/4}$ and $\kappa_v \geq a^{\frac{a^{1/4} + \epsilon}{4n}} \kappa_v(t)$
\end{enumerate}
\end{lemma}
\begin{proof}
Let $\epsilon > 0$, denote by $R$ the radius beyond which $\abs{D}\leq \epsilon$, and assume $\sqrt{a}<1$.
The proof proceeds by establishing the existence of regions in $[r_0,r_1]$ where both $N$ and $N_v$ are less or equal to $\sqrt{a}$ at the same time.
The stated bounds then follow from \autoref{lem:dichotomy-frequency-function-bounds} (with $a$ replaced by $\sqrt{a}$ everywhere).

First, regarding the condition $N\leq\sqrt{a}$, observe that the set ${I = \left\{\ r\in[\log r_0, \log r_1] \mid N(\exp r ) > \sqrt{a}\ \right\}}$ makes up at most $\sqrt{a}$ of the length of the surrounding interval.
To see this, write $I = \coprod (a_k, b_k)$ and then go from $\kappa(r_0)$ to $\kappa(r_1)$ by iteratively using \autoref{cor:dichotomy-kappa-asymptotic-grönwall}, with bounds $N>\sqrt{a}$ on each $(a_k, b_k)$ and $N\geq 0$ on the intervals $[b_k, a_{k+1}]$ in between.
This leads to
\begin{align}
	\kappa(r_1) \geq \left(\frac{r_1}{r_0}\right)^{-\epsilon} \prod_k \left(\frac{b_k}{a_k}\right)^{\sqrt{a}} \kappa(r_0)\ .
\end{align}
This inequality is only compatible with the assumption that ${\kappa(r_1) \leq (r_1/r_0)^{a-\epsilon} \kappa(r_0)}$ if
\begin{align}
	\sum_k (\log b_k - \log a_k) \leq \sqrt{a}\; (\log r_1 - \log r_0)\ .
\end{align}
Equivalently, if $\abs{I} \leq \sqrt{a} \abs{[\log r_0, \log r_1]}$.

An analogous statement for points that satisfy the condition $N_v\leq \sqrt{a}$ makes use of a slightly longer argument.
As before, we set out to investigate the measure of the set on which $N_v > \sqrt{a}$.
Denote by $L$ the largest integer such that $2^L r_0 < r_1$ and consider the sequence $\{ 2^\ell r_0\}_{\ell=0, 1, \ldots, L-1}$.
For each pair of neighbours in this sequence we can apply equation \eqref{eq:dichotomy-smallest-eigenvalue-is-almost-nondecreasing} in the form
\begin{align}
	\lambda(2s) \geq \lambda(s) + \frac{1}{M} \sum_{k=1}^M \lambda(s_k) \left(N_{v_k}(r_k) - \epsilon \right) + \mathcal{O}\left(\frac{1}{M^2}\right)\ ,
	\label{eq:dichotomy-correlation-tensor-eigenvalue-comparison}
\end{align}
where $s_k = (1+k/M)\: s$ and $v_k$ denotes an eigenvector of $T(s_k)$ with eigenvalue $\lambda(s_k)$.
Next, note that \autoref{cor:dichotomy-kappa-asymptotic-grönwall} with the lower bound $N\geq 0$ provides
\begin{align}
	\lambda(s_k)
	= \kappa^2_{v_k}(s_k)
	\geq \left(\frac{s_k}{s}\right)^{-2\epsilon} \kappa^2_{v_k}(s)
	\geq \frac{1}{4} \lambda(s)\ .
\end{align}
Here the last inequality holds by virtue of $s_k/s < 2$ and because $\lambda(s)$ is the smallest eigenvalue of $T$ at $s$, such that $\kappa^2_{v_k}(s) = T(s)(v_k,v_k)$ can't be smaller.
Also introduce for the moment the notation $\langle N_v \rangle_{\ell}$ for the average of $\{ N_{v_k}(s_k) \}_{k=1,\ldots,M}$ in the interval $[2^\ell r_0, 2^{\ell+1} r_0]$.
The inequality in \eqref{eq:dichotomy-correlation-tensor-eigenvalue-comparison} becomes
\begin{align}
	\frac{\lambda(2^{\ell+1}r_0)}{\lambda(2^\ell r_0)} \geq 1 + \frac{1}{4}\bigg( \langle N_v \rangle_{\ell} - \epsilon \bigg) + \mathcal{O}\left(\frac{1}{M^2}\right)\ .
\end{align}
In the last of these pairings we can use $r_1$ as endpoint of the interval instead of $2^L r_0$ without changing the inequality, since $r_1 - 2^{L-1}r_0 > 2^{L-1} r_0$ and $r_1 \leq 2^{L+1} r_0$.
Hence, the product of all these ratios yields
\begin{align}
	\frac{\lambda(r_1)}{\lambda(r_0)} \geq \prod_{\ell = 0}^{L-1} \left( 1 + \frac{\langle N_v \rangle_{\ell} - \epsilon}{4} + \mathcal{O}\left(\frac{1}{M^2}\right) \right)\ .
\end{align}
Let $f$ denote the fraction of $\{{0, 1, \ldots, L-1\} }$ for which $\langle N_v \rangle_\ell + \mathcal{O}(1/M^2) > \sqrt{a}$.
Then the product simplifies to
\begin{align}
	\frac{\lambda(r_1)}{\lambda(r_0)} \geq \left( 1 + \frac{\sqrt{a} - \epsilon}{4} \right)^{fL} \left( 1 - \frac{\epsilon}{4} + \mathcal{O}\left(\frac{1}{M^2}\right) \right)^{1-fL}\ .
\end{align}
On the left use that $\lambda(r_1)\leq \kappa^2(r_1) \leq \left(\frac{r_1}{r_0}\right)^{2(a-\epsilon)} \kappa^2(r_0)$ to replace $\lambda(r_1)$.
Upon taking logarithms on both sides the resulting inequality reads
\begin{align}
	2(a-\epsilon) \log \left(\frac{r_1}{r_0} \right) + \log\left(\frac{\kappa^2(r_0)}{\lambda(r_0)} \right)
	\geq fL \log\left( 1+ \frac{\sqrt{a}-\epsilon}{4}\right) + (1-f)L \log \left( 1 - \frac{\epsilon}{4} + \mathcal{O}\left(\frac{1}{M^2}\right) \right)\ .
\end{align}
This can be simplified in several ways:
\begin{enumerate*}[label={(\arabic*)}]
	\item Because $2^{L+1} r_0 > r_1$ and as long as we make sure that $r_1 \geq 4 r_0$ we can assume $L>\frac{1}{2} \log r_1/r_0$.
	\item Since $\sqrt{a}-\epsilon < 1$ the term $\log\left(1+(\sqrt{a}-\epsilon)/4 \right)$ is no smaller than $(\sqrt{a}-\epsilon)/8$.
	\item By taking $M$ large, we can make sure that $\log\left( 1 - \epsilon/4 + \mathcal{O}(1/M^2) \right)$  is larger than $- \epsilon$.
\end{enumerate*}
Plugging everything in and solving for $f$ leads to the following upper bound
\begin{align}
f < 32 \sqrt{a} \left(1 + \frac{\log\left(\kappa^2(r_0)/\lambda(r_0)\right)}{2a \log \left(r_1/r_0 \right) } \right)\ .
\end{align}
In particular, if $r_1 > \left(\kappa^2(r_0) / \lambda(r_0) \right)^{\frac{1}{2a}} r_0$ we find that $f<64 \sqrt{a}$.

Consider the subset $J\subset [\log r_0, \log r_1]$ that consists of those intervals ${[\log 2^\ell r_0, \log 2^{\ell+1} r_0]}$ on which $\langle N_v \rangle_\ell$ is larger than $\sqrt{a}$.
The inequality for $f$ now tells us that the length of $J$ can be at most $64 \sqrt{a}\; \abs{[\log r_1, \log r_0]}$.
This means that the union of $J$ with the interval $I$ from the first part has bounded measure
\begin{align}
	\abs{I\cup J} < 65 \sqrt{a}\; \abs{[\log r_0, \log r_1]}\ .
\end{align}
We conclude that whenever $\sqrt{a}$ is sufficiently small (for example $\sqrt{a}<0.01$), then there exists a point $\log(t_1)$ that is neither in $I$ nor in $J$.
In fact, if $r_0 < r_1^{1-100\sqrt{a}}$, we can moreover find a $t_1$ that is larger than $r_1^{1-100\sqrt{a}}$, since the measure of $[r_1^{1-100\sqrt{a}}, r_1]$ is still larger than $I$ and $J$ combined.

Choose and fix a $t_1$ with these properties.
On the one hand this means that $N(t_1) \leq \sqrt{a}$ and \autoref{lem:dichotomy-frequency-function-bounds} implies that on the interval $[\tilde R_1, t_1]$ we have $N < a^{1/4}$ and ${\kappa \geq a^{\frac{a^{1/4}+\epsilon}{4n}} \kappa(r)}$.
Here, $\tilde R_1$ is the larger of $a^{1/4n} t_1$ and $R$.
On the other hand it means that $t_1$ is from an interval $[2^\ell r_0, 2^{\ell+1} r_0]$ on which the average value $\langle N \rangle_\ell$ is less than $\sqrt{a}$.
Consequently, the partition of that interval contains a radius $t_2$ at which $N_{v}(t_2)<\sqrt{a}$.
We again deduce that $N_v < a^{1/4}$ and ${\kappa_v \geq a^{\frac{a^{1/4} + \epsilon}{4n}} \kappa_v(r)}$ on $[\tilde R_2, t_2]$, where $\tilde R_2$ is the larger of $a^{1/4n} t_2$ and $R$.

Finally, $[\tilde R_1, t_1]$ and $[\tilde R_2, t_2]$ intersect whenever $a$ is small enough.
To see this explicitly, assume w.l.o.g that $t_1 > t_2$ and note that $t_1/t_2 \leq 2$ since both are contained in an interval of the form $[2^\ell r_0, 2^{\ell +1} r_0]$.
It follows that ${\tilde R_1 / t_2 < 2 a^{1/4n} }$, which is less than one if for example $a^{1/4n}<0.1$.
The upper bound of $0.1$ is convenient, because we then also have $\tilde R_1 < a^{1/8n} t_2$, such that choosing $t=t_2$ and ${\tilde R = \max( a^{1/8n} t, R)}$ concludes the proof.
\end{proof}


\section{A Priori Bounds}
\label{sec:dichotomy-a-priori-bounds}

The final ingredient for the proof of \autoref{bigthm:dichotomy} are pointwise a priori bounds for $\norm{\phi}$ in terms of $\kappa$, as well as estimates for the volume of the subsets of $B_r$ on which $\norm{\phi}$ is small compared to that upper bound.
These are the promised analogues of the mean value inequality.

The proof relies mainly on the standard approach to the mean-value inequality and the fact that according to \autoref{prop:dichotomy-derivative-of-kappa}, $\kappa$ can't decrease too rapidly at infinity.
However, due to the asymptotic nature of the latter statement and the related lack of control in the interior, these bounds only hold for points that lie in some large spherical geodesic shell surrounding $p$.

Below we use the notation $S(r_1,r_2)$ for the closed spherical geodesic shell based at $p$ with inner radius $r_1$ and outer radius $r_2$, i.e. all points that satisfy $r_1 \leq d(p,x) \leq r_2$.
Moreover, we write $f_{\Omega_i} := \frac{\vol \Omega_i}{\vol \Omega}$ for the fraction that a subset $\Omega_i$ occupies within its corresponding surrounding set~$\Omega$.

The following result holds verbatim if we replace $\phi$, $\kappa$ and $N$ by the versions $\phi_v$, $\kappa_v$ and $N_v$ associated to a unit vector $v \in T_p W$.
\begin{lemma}\label{lem:dichotomy-a-priori-bounds}
Let $W^n$ be an ALX$_k$ space of dimension $n$, $p\in W^n$, $\Ric \geq 0$, and assume $(A,\phi)$ is a solution of~\eqref{eq:dichotomy-main-assumption}.
For any $\epsilon\in (0,1/2)$ there exists a radius $R\geq 0$ and constants $c_i>1$ that only depend on $n$ and $k$, such that for sufficiently large $r \geq R$ the following holds.
\begin{enumerate}
	\item $\norm{\phi(x)} \leq c_0\; \frac{\kappa(r)}{{\sqrt{\vol X}}}$ for any $x \in S\, \bigl( R+\frac{r-R}{8} ,\: r-\frac{r-R}{8} \bigr)$,
	\item Let $t \in [R +\frac{6}{8}(r-R), R+\frac{7}{8}(r-R)]$ and assume that $N\leq 1$ on $[t,r]$.
	Then the subset $\Omega_1 \subseteq \del B_t$ on which ${ \norm{\phi(x)} \leq \frac{1}{2} \frac{\kappa(r)}{\sqrt{\vol X}}}$ has relative volume
	$f_{\Omega_1} \leq c_1 \epsilon + c_2 \sqrt{N} \ .$
\end{enumerate}
\end{lemma}

\begin{proof}
Throughout the proof we fix some $\epsilon \in (0,1/2)$ and choose $R$ to be some large enough radius, such that on the one hand $\kappa$ is $\epsilon$-almost non-decreasing as provided by \autoref{prop:dichotomy-derivative-of-kappa} and on the other hand $\vol B_r \leq (1+\epsilon) \vol X\: r^{n-k}$ for any $r\geq R$ as in \autoref{prop:dichotomy-alx-volume-growth}.
Assume for now simply that $r > R$ is some fixed outer radius.
We will collect conditions on $r$ as we go and will find that they can always be met by choosing some larger $r$ to start with.

\textit{ad (i)}\\[0.4em]
In this part of the proof we use a specific bump function $\beta$ on $W^4$ with compact support inside the geodesic shell $S(R,r)$, constructed as follows.
Denote by $\beta_{\mathbb{R}}$ a non-increasing function on $\R$ that is equal to $1$ on $\left(-\infty, \frac{14}{16} \right)$ and equal to $0$ on $\left( \frac{15}{16} , \infty \right)$.
Use this to define a corresponding function on $W^4$ by the rule
\begin{align}
	\beta(x) = \beta_{\mathbb{R}} \left(\, \abs{\, \frac{ d(p,x) - \frac{r+R}{2} }{ \frac{r-R}{2}} \, } \, \right)\ .
\end{align}
The function $\beta$ vanishes on the inner ball ${ B_{R+\frac{r-R}{32}}(p) }$, is equal to $1$ on the geodesic shell with inner radius ${R+\frac{r-R}{16}}$ and outer radius $r- \frac{r-R}{16}$, and is zero again outside of ${ B_{r-\frac{r-R}{32}}(p) }$.
Note in particular that $\beta$ and its derivatives have compact support in the interior of $S(R,r)$.
Furthermore, the gradient and Laplacian of $\beta$ are bounded as follows
\begin{align}
	\abs{d\beta} &\leq \frac{32}{(r-R)} & \abs{\Delta \beta} &\leq \frac{32}{(r-R)^2}\ .
\end{align}

Denote by $G_x$ the positive Dirichlet Green's function of the Laplacian on $B_r(p)$ based at $x\in B_r(p)$.
Recall from \autoref{lem:dichotomy-bound-for-greens-function} that, due to the volume growth of $W^n$, for large enough distances the Green's function and its derivative are bounded from above as follows
\begin{align}
	G_x(y) &\leq \frac{(1+\epsilon) c}{\vol X\: d(x,y)^{n-k-2}} & \abs{dG_x(y)} &\leq \frac{(1+\epsilon) c}{ \vol X\: d(x,y)^{n-k-1}}\ ,
\end{align}
where $c$ depends only on the dimension $n$.

With that in mind we now note that contracting equation \eqref{eq:dichotomy-main-assumption} with $\langle \cdot, \phi \rangle$ leads to
\begin{align} \label{eq:dichotomy-main-assumption-contracted}
	\frac{1}{2} \Delta_B \norm{\phi}^2 + \lVert \nabla^A \phi \rVert^2 + \norm{[\phi\wedge\phi]}^2 + \langle \Ric \phi, \phi \rangle = 0\ ,
\end{align}
where $\Delta_B = \nabla^\dagger \nabla$ is the Bochner Laplacian associated to the Levi-Civita connection.
This equation implies $\Delta_B \norm{\phi}^2 \leq 0$, so the function $\norm{\phi}^2$ is {subharmonic\footnotemark} and accordingly satisfies a version of the mean-value inequality.
\footnotetext{Note that the Bochner and connection Laplacian differ by a sign: $\Delta_B = -\tr \nabla^2$.}
To see this directly, multiply \eqref{eq:dichotomy-main-assumption-contracted} by $\beta G_x$ and integrate over $B_r(p)$
\begin{align}
	\int_{B_r(p)}\ \beta G_x \left( \Delta_B \norm{\phi}^2 + \lVert \nabla^A \phi \rVert^2 + \norm{[\phi\wedge\phi]}^2 + \langle \Ric \phi, \phi \rangle \right) = 0\ .
\end{align}
Upon integration by parts in the first term, using that $\beta G_x  = 0$ on $\del B_r(p)$, and assuming that $x$ is an element of the geodesic shell on which $\beta = 1$, we find
\begin{align}
	\norm{\phi(x)}^2
	+ \int_{B_r(p)}\ \left( g^{-1}( d\beta, d G_x ) + (\Delta_B \beta) G_x \right) \norm{\phi}^2
	+ \int_{B_r(p)}\ \beta G_x  \left( \lVert \nabla^A \phi \rVert^2 + \norm{[\phi\wedge\phi]}^2 + \langle \Ric \phi, \phi \rangle \right)
	= 0\ .
\end{align}
Since $\Ric\geq 0$, the last term on the left hand side is non-negative, so the equation provides the following estimate.
\begin{align}
	\norm{\phi(x)}^2 \leq \abs{ \int_{B_r(p)}\ \left( g^{-1}( d G_x, d\beta ) +  G_x \Delta_B \beta \right) \norm{\phi}^2 }
\end{align}
Assume ${x \in S(R+\frac{r-R}{8}, r-\frac{r-R}{8})}$, such that the distance from $x$ to the support of $d\beta$ and $\Delta \beta$ is greater or equal to $\frac{r-R}{16}$.
Furthermore, make $r$ large enough that the previously mentioned bounds on $G_x(y)$ hold for all points with distance $d(x,y) \geq \frac{r-R}{16}$.
Then, using the Cauchy-Schwarz inequality on the first term and the estimates for $G$, $\abs{dG}$, $\abs{d\beta}$ and $\abs{\Delta_B\, \beta}$, we arrive at
\begin{align}
	\norm{\phi(x)}^2
	\leq \frac{32 (1+\epsilon) c}{\vol X\: (r-R)^{n-k}} \int_{S(R,r)} \norm{\phi}^2\ .
\end{align}
As final step use that in the given domain of integration $\kappa$ is $\epsilon$-almost non-decreasing.
Hence, \autoref{cor:dichotomy-kappa-asymptotic-grönwall} with $N\geq 0$ provides the estimate $\kappa(t) \leq \left(\frac{r}{t} \right)^\epsilon \kappa(r)$ for all $t\leq r$.
The integral in the last equation is thus bounded from above by
\begin{align} \label{eq:dichotomy-integration-of-kappa}
	\int_{S(R,r)} \norm{\phi}^2
	= \int_R^r dt\ t^{n-k-1} \kappa^2(t)
	\leq \frac{r^{n-k}}{n-k-2\epsilon}\  \kappa^2(r)\ .
\end{align}
Plugging this in, assuming $r>2R$, and using $\epsilon< 1/2$ leads to
\begin{align}
	\norm{\phi(x)}^2 \leq \frac{2^{n-k+6} c}{(n-k-1)\, \vol X } \kappa^2(r) \leq c_0^2 \frac{\kappa^2(r)}{\vol X}\ .
\end{align}
We note that $c_0$ only depends on the dimension of $W^n$ and the dimension of the fibers at infinity.

\vspace{1em}
\textit{ad (ii)}\\[0.4em]
Fix some radius $t \in [R+\frac{6}{8}(r-R), R+\frac{7}{8}(r-R)]$ and consider the associated geodesic sphere $\del B_t$.
We are interested in the volume of the subset of $\del B_t$ where $\norm{\phi}$ is small compared to the bound in \textit{(i)}.

Hence, write $\Omega_1 \subseteq \del B_t$ for the points where ${\norm{\phi(x)} \leq \frac{\kappa(r)}{2\sqrt{\vol X}} }$.
Also introduce $\Omega_2 \subseteq \del B_t$ to denote the set of points at which $\norm{\phi(x)} \leq (1+\sqrt{N(r)}) \frac{\kappa(r)}{\sqrt{\vol X}}$.
Since $N\geq 0$, $\Omega_2$ contains $\Omega_1$.

Split up the contributions to $\kappa^2(t)$ that arise from integration over $\Omega_1$, $\Omega_2\setminus\Omega_1$, and their complement $\Omega_3 = \del B_t \setminus \Omega_2$.
\begin{align}
	\kappa^2(t) = \frac{1}{t^{n-k-1}} \left( \int_{\Omega_1} \norm{\phi}^2 + \int_{\Omega_2\setminus\Omega_1} \norm{\phi}^2 + \int_{ \Omega_3 } \norm{\phi}^2  \right)
\end{align}
On $\Omega_1$ the integrand is bounded by $\frac{\kappa^2(t)}{4 \vol X}\, $, on $\Omega_2\setminus\Omega_1$ we use $(1+\sqrt{N})^2 \frac{\kappa^2(r)}{\vol X}$, and the integral over $\Omega_3$ can't be larger than $\kappa^2(t)$ in any case.

With regard to the last of these bounds we now make use of the fact that $\kappa$ is almost non-decreasing, which allows us to compare $\kappa^2(t)$ to $\kappa^2(r)$.
To that end consider the derivative of $\kappa^2$.
\begin{align}
	\Restrict{\frac{d\kappa^2}{dr}}{\tilde r}
	= 2 \frac{N+D}{\tilde r} \kappa^2(\tilde r)
	\geq - 2\epsilon \frac{\tilde{r}^{1+2\epsilon}}{r^{2+2\epsilon}} \kappa^2(r)
\end{align}
For the estimate on the right hand side we have used on the one hand that $N\geq 0$ and $D \geq -\epsilon$, and on the other hand that $N\leq 1$ such that \autoref{cor:dichotomy-kappa-asymptotic-grönwall} implies that $\kappa^2(\tilde r) \geq (\tilde r / r)^{2+2\epsilon} \kappa^2(r)$.
Integration from $t$ to $r$ leads to
\begin{align}
	\kappa^2(t) \leq \frac{\epsilon}{1+\epsilon} \left( 2 - \left( \frac{t}{r} \right)^{2+2\epsilon} \right) \kappa^2(r)\ .
\end{align}
Since $t \leq r$, we may as well use the somewhat simpler statement $\kappa^2(t) \leq 2 \epsilon \kappa^2(r)$.

Plugging in the corresponding bounds on each of the $\Omega_i$, writing $\vol \Omega_i = f_{\Omega_i} \vol \del B_t$, $f_{\Omega_2} \leq 1$, and using $\vol \del B_t \leq (1+\epsilon) \vol X t^{n-k-1}$ thus leads to
\begin{align}
	\kappa^2(t)
	&\leq \left( \frac{1}{4} f_{\Omega_1} + (1-f_{\Omega_1})\ (1 + \sqrt{N(r)}\ )^2 + 2 \epsilon \right) (1+\epsilon) \kappa^2(r) \\
	&\leq \left( 1 + 2 \epsilon - \frac{3}{4} f_{\Omega_1} \right)  ( 1 + \sqrt{N(r)}\ )^2 (1+\epsilon) \kappa^2(r)\ ,
\end{align}
which can be rearranged to
\begin{align} \label{eq:dichotomy-size-of-set-where-phi-is-small}
	f_{\Omega_1}
	&\leq \frac{4}{3} \left( 1 + 2 \epsilon - \frac{\kappa^2(t)}{ \bigl( 1 + \sqrt{N(r)} \, \bigr)^2\, (1+\epsilon) \kappa^2(r) } \right)\ .
\end{align}

To make this expression more useful we'll now also need a lower bound for $\kappa^2(t)$ in terms of $\kappa^2(r)$.
This can again be achieved by considering the derivative of $\kappa^2$.
In this case we observe that the following function is non-decreasing in $\tilde{r}$.
\begin{align}
	\tilde r \mapsto \tilde{r}^{n-k-2}\kappa^2(\tilde r) N(\tilde r) = \int_{B_{\tilde{r}}(p)} \lVert \nabla^A \phi \rVert^2 + \norm{[\phi\wedge\phi]}^2 + \langle \Ric(\phi),\phi \rangle
\end{align}
Moreover, since $\kappa$ is $\epsilon$-almost non-decreasing for $\tilde r > R$, we have $\abs{D} \leq N+\epsilon$.
This yields the following upper bound for the derivative of $\kappa^2$:
\begin{align}
	\Restrict{ \frac{d \kappa^2}{dr} }{\tilde{r}}
	= 2 \frac{N + D}{\tilde{r}} \kappa^2(\tilde{r})
	\leq 4\: \frac{r^{n-k-2}}{\tilde{r}^{n-k-1}} N(r) \kappa^2(r) + 2\epsilon\: \frac{r^{2\epsilon}}{\tilde{r}^{1+2\epsilon}}\kappa^2(r)\ .
\end{align}
Integration\footnotemark~from $t$ to $r$ now leads to
\footnotetext{%
For notational simplicity we assume here that $k\neq n-2$.
However, the result holds similarly for the case $k=n-2$, where the only difference is that upon integration the formulas contain logarithms.
}
\begin{align}
	\kappa^2(r) - \kappa^2(t)
	&\leq \frac{4}{n-k-2} \left(\frac{r}{t}\right)^{n-k-2}  N(r) \kappa^2(r) + \left( \left(\frac{r}{t}\right)^{2\epsilon} - 1 \right) \kappa^2(r)\ .
\end{align}
Recall that $r/t < 4/3$ and observe that thus ${(r/t)^{2\epsilon} - 1 < \epsilon}$ for any choice of $\epsilon \in (0, 1/2)$.
It follows that
\begin{align}
	\kappa^2(t) \geq \Bigl( 1 - \epsilon - c N(r) \Bigr)\: \kappa^2(r)\ ,
\end{align}
where the constant $c$ only depends on $n$ and $k$.
Plugging this lower bound for $\kappa^2(t)$ into \eqref{eq:dichotomy-size-of-set-where-phi-is-small} and using that $\epsilon, N<1$, we conclude that
\begin{align}
	f_{\Omega_1}
	&\leq c_2 \epsilon + c_3 \sqrt{ N(r) }\ ,
\end{align}
where $c_2$ and $c_3$ only depend on $n$ and $k$.
\end{proof}


\section{Proof of Taubes' Dichotomy on ALX spaces}
\label{sec:dichotomy-proof}

We are now in a position to prove \autoref{bigthm:dichotomy}.
As advertised before, the arguments are in complete analogy to Taubes' original proof \cite{Taubes2017a}.

\begin{theorem} \label{thm:dichotomy}
Let $W^n$ be an ALX$_k$ gravitational instanton of dimension $n\geq 2$ with asymptotic fibers of dimension $k \leq n-1$ and fix a point $p \in W^n$.
Assume $(A,\phi)$ satisfies the second-order differential equation~\eqref{eq:dichotomy-main-assumption}.
Then
\begin{enumerate}
	\item either there is an $a>0$ such that $\liminf_{r\to\infty} \frac{\kappa(r)}{r^a} > 0$,
	\item or $[\phi \wedge \phi]=0$.
\end{enumerate}
\end{theorem}

\begin{proof}
It is sufficient to consider the case where $\kappa$ is not asymptotically zero, since otherwise $\kappa$ -- and thus $\phi$ -- has compact support due to the (asymptotically) unique continuation property of \autoref{lem:dichotomy-unique-continuation} and in that case $[\phi\wedge\phi]$ must vanish.
To see this, assume to the contrary that $[\phi\wedge\phi]$ is non-zero on some neighbourhood of a point $p$.
Since $\phi$ is compactly supported, there is some radius $R$ such that $\phi$ vanishes on $W^n \setminus B_{R}(p)$ and $\kappa$ vanishes on $[R,\infty)$.
As a consequence, there is some positive constant $c$ such that for any $r>R$ we have
\begin{align}
	\int_{B_r(p)} \nu \geq \int_{B_r(p)} \norm{[\phi\wedge\phi]}^2 = c > 0\ .
\end{align}
Here, $\nu = \lVert \nabla^A \phi \rVert^2 + \norm{[\phi\wedge\phi]}^2 + \langle \Ric(\phi),\phi \rangle$ denotes the integrand that defines $N$.
The derivative of $\kappa^2$ at any $r > R$ is given by
\begin{align}
	\frac{d \kappa^2}{dr} = \frac{1}{r^{n-k-1}} \left( \int_{B_r(p)} \nu + \int_{\del B_r(p)} \left( \Delta r - \frac{n-k-1}{r} \right) \norm{\phi}^2 \right)\ .
\end{align}
The second integral vanishes, since $\phi=0$ at points with $d(x,p) > R$.
However, the first term is bounded from below by $\frac{c}{r^{n-k-1}} > 0$, which is in contradiction to the assumption that $\kappa$ is constant (namely $0$) on all of $[R, \infty)$.

By the same reasoning one also finds that $[\phi, \phi_v] =0$ whenever $\kappa_v$ has compact support.
As a consequence it is sufficient to consider the components of $\phi$ that are in the complement of the zero eigenspace of $T$ at infinity (see the related discussion in \autoref{sec:dichotomy-correlation-tensor}).
Note in particular that when only a single component $\phi_v$ is non-zero at infinity, then $[\phi\wedge\phi]=0$ everywhere.
Hence, assume from now on that $T$ acts on a vector space of dimension at least 2 and that $\lambda>0$ on all of $(0,\infty)$.

To prove the dichotomy stated in the theorem, assume that $\kappa$ is not asymptotically bounded below by any positive power of $r$.
This means that for any small $a>0$ (say, for example, small enough that ${a^{1/8n}< 0.1}$) we can do the following:
Set $\epsilon = a/2$ and let $R>1$ be a radius beyond which $\abs{D}\leq \epsilon$ and $\vol B_r(p) \leq (1+\epsilon) \vol X r^{n-k}$.
Then we can find an arbitrarily large $r_1 \in [R,\infty)$ such that ${\kappa(r_1)\leq \left(\frac{r_1}{R}\right)^{a-\epsilon} \kappa(R)}$.
In particular we may choose some $r_1$ that is larger than each of the four numbers $4 R$, $\left( \kappa^2(R)/\lambda(R) \right)^{1/2a} R$, $\kappa(R)^{-1/a}$, and $R^{1/({1-100\sqrt{a}} )}$.
We are then in the situation in which we can rely on \autoref{lem:dichotomy-weighted-frequency-function-bounds}.
The arguments in the upcoming five parts show that this leads to a contradiction if $N\neq 0$ and $a$ is too small.

\paragraph{Part 1}
Recall that \autoref{lem:dichotomy-weighted-frequency-function-bounds} provides the existence of a distinguished radius $t \in [r_1^{1-100 \sqrt{a}}, r_1] \subseteq [R,r_1]$.
In this part we collect our previous results and slightly expand on our knowledge about the eigenvalues of $T$ at and below $t$.

Write $u$ and $v\/$ for unit eigenvectors associated to the \emph{largest} and \emph{smallest} eigenvalue of $T(t)$, respectively.
Recall that these eigenvalues coincide at $t$ with the values of $\kappa_u^2(t)$ and $\kappa_v^2(t)$, so they satisfy ${\kappa_v^2(t)\leq \kappa_u^2(t)}$.
If $\kappa_v^2(t) \neq \kappa_u^2(t)$ the eigenvectors $v$ and $u$ are guaranteed to be orthogonal.
Otherwise $T(t)$ is a multiple of the identity matrix and we choose an arbitrary pair of orthonormal vectors.

Denote by $\tilde R$ the larger of $a^{\frac{1}{8n}} t$ and $R$.
\autoref{lem:dichotomy-weighted-frequency-function-bounds} establishes that on the interval $I:=[\tilde R, t]$ the frequency functions $N$ and $N_v$ are bounded from above by $a^{1/4}$ and provides associated lower bounds for $\kappa$ and $\kappa_v$.
As we show now, analogous estimates hold for $N_u$ and $\kappa_u$.

First observe that the largest eigenvalue satisfies $\kappa^2_u(t) \geq \frac{1}{n}\kappa^2(t)$ since $\kappa^2$ is the trace of $T$.
From this and the definition of $N$ and its $N_u$ version (cf. \autoref{prop:dichotomy-derivative-of-kappa} and \eqref{eq:dichotomy-weighted-frequency-function}, respectively) it follows that ${N_u(t) \leq n N(t)}$.
As a consequence $N_u \leq n a^{1/4}$ and, as long as we make sure that $a^{1/4} < 1/n$, a small variation of the second part of \autoref{prop:dichotomy-small-N} finds that $N_u \leq a^{1/8}$ on all of $I$.
Since $N_u$ is bounded from above, we can now deduce bounds on $\kappa_u$ as usual via \autoref{cor:dichotomy-kappa-asymptotic-grönwall}.
\begin{align}
	\kappa_u \geq a^{\frac{a^{1/8}+\epsilon}{8n}} \kappa_u(t)\ .
\end{align}
In the current situation with $\epsilon = a/2$ this bound and its analogue for $\kappa_v$ can be simplified considerably.
For that observe that $a^{\epsilon} = a^{a/2}$ is certainly larger than $1/2$, and similarly $a^{\frac{a^{1/4}}{4n}} > (1/2)^{\frac{4}{4n}} > 1/2$ and $a^{\frac{a^{1/8}}{8n}} > 1/2$.
Applying these observations to the bounds for $\kappa_u$ and $\kappa_v$ results in the main conclusions of this part.

The following estimates hold on all of $I=[\tilde R, t]$:
\begin{itemize}
	\item $N_v \leq a^{1/4}$ and $\kappa_v \geq \frac{1}{4} \kappa_v(t)$
	\item $N_u \leq a^{1/8}$ and $\kappa_u \geq \frac{1}{4} \kappa_u(t)$
\end{itemize}

\paragraph{Part 2}
The goal of this part is to show that there exist points in  ${J= \bigl[ \tilde R + \frac{6}{8}(t-\tilde R), \tilde R + \frac{7}{8}(t-\tilde R) \bigr]}$ for which the integrals that appear in $N_u$ and $N_v$ are both small.
This interval is of significance, because it corresponds to radii that are contained in the geodesic shell that appears in item \textit{(iii)} of \autoref{lem:dichotomy-a-priori-bounds}.

Denote by $\nu_u  := \left( \lVert \nabla^A \phi_u \rVert^2 + \norm{[\phi\wedge\phi_u]}^2 \right)$ the integrand in $N_u$ and analogously for $\nu_{v\/}$.
Consider the following sets of radii in $J$.
\begin{align}
	J_u &:= \set{ s \in J \suchthat \int_{\del B_s} \nu_u  \leq \frac{1}{2 \abs{J}}\; t^{n-k-2}\; \kappa_u^2(t)\; a^{1/16} }
	\label{eq:dichotomy-main-proof-N-boundary-integral1} \\
	J_v &:= \set{ s \in J \suchthat \int_{\del B_s} \nu_v  \leq \frac{1}{2 \abs{J}}\; t^{n-k-2}\; \kappa_v^2(t)\; a^{1/16} }
	\label{eq:dichotomy-main-proof-N-boundary-integral2}
\end{align}
Note for later that $\abs{J} = \abs{I}/8$, where $I = [\tilde R, t]$, so the fact that $a^{1/8n} t \leq \tilde R$ and using $a^{1/8n} < 0.5$ yields $\abs{J} \geq \frac{t}{16}$.

The measure of $J_u$ is bounded from below, since
\begin{align}
	t^{n-k-2}\kappa^2_u(t) N_u(t)
	\geq \int_J ds \int_{\del B_s} \nu_u
	\geq \left( \abs{J} - \abs{J_u} \right)\; \frac{1}{2 {\abs{J}}}\; t^{n-k-2}\; \kappa_u^2(t)\; a^{1/16}\ ,
\end{align}
and similarly for $J_v$.
Since both $N_u(t), N_v(t) \leq a^{1/8}$, we find
\begin{align}
	\abs{J_u}, \abs{J_v} > \left( 1 - 2 a^{1/16} \right) \abs{J}\ .
\end{align}
Since $2 a^{1/16} < 0.5$ (recall that $a^{1/8n}<0.1$ and $n\geq 2$ in any case), it follows that $J_u$ and $J_v$ must intersect.
Hence, choose and fix from now on a point $r \in J$ at which both \eqref{eq:dichotomy-main-proof-N-boundary-integral1} and \eqref{eq:dichotomy-main-proof-N-boundary-integral2} are satisfied.

\paragraph{Part 3}
As a next step we establish an $L^2$-bound for the function $\Tr \phi_u \phi_v$ on $\del B_r$, where $r$ is the fixed radius we found at the end of Part 2.
For this we view $\del B_r$ with the induced metric as a compact Riemannian manifold, where we can rely on a Hölder-Sobolev inequality.

First note that the derivative of $\Tr \phi_u\phi_v$ is bounded by
\begin{align}
	\norm{d \Tr \phi_u \phi_v } \leq \lVert \nabla^A \phi_u \rVert\; \lVert \phi_v \rVert + \lVert \phi_u \rVert\; \lVert \nabla^A \phi_v \rVert\ .
\end{align}
Because $r$ lies inside the geodesic shell of \autoref{lem:dichotomy-a-priori-bounds} we know that $\norm{\phi_u} \leq c_0 \kappa_u(t)$ and $\norm{\phi_v}\leq c_0 \kappa_v(t)$.
In light of the fact that both \eqref{eq:dichotomy-main-proof-N-boundary-integral1} and \eqref{eq:dichotomy-main-proof-N-boundary-integral2} are satisfied at $r$ (and using the Cauchy-Schwarz inequality), we find that
\begin{align}
	\int_{\del B_r} \norm{d \Tr \phi_u \phi_v}^2 \leq 16 c_0^2 t^{n-k-3} \kappa_u^2(t) \kappa_v^2(t) a^{1/16}\ .
\end{align}

On the $n-1$ dimensional compact manifold $\del B_r$ the Gagliardo-Nirenberg-Sobolev inequality holds, such that
\begin{align}
	\norm{ \Tr \phi_u\phi_v }_{L^q(\del B_r)} \leq C\, \norm{d \Tr \phi_u \phi_v}_{L^2(\del B_r)}\ ,
\end{align}
where $\frac{1}{q} = \frac{1}{2}- \frac{1}{n-1}$ and the constant $C$ only depends on $n$.
On the other hand, since $\del B_r$ has finite measure we can also use Hölder's inequality with $\frac{1}{p} = \frac{1}{q} + \frac{1}{n-1} (= \frac{1}{2})$, such that we obtain
\begin{align}
	\norm{ \Tr \phi_u\phi_v }_{L^2(\del B_r)}
		\leq \operatorname{meas}(\del B_r)^{\frac{1}{n-1}} \norm{ \Tr \phi_u\phi_v }_{L^q(\del B_r)}
		\leq C_{HS}\, r \norm{d \Tr \phi_u \phi_v}_{L^2(\del B_r)}\ ,
\end{align}
where we have used Bishop-Gromov's volume comparison \autoref{thm:dichotomy-bishop-gromov} to estimate the area of the geodesic sphere and absorbed any radius-independent contributions into the Hölder-Sobolev constant $C_{HS}$, which then still depends only on $n$.

More explicitly, we conclude that we have the following $L^2$-bound for $\Tr\phi_u\phi_v$ over $\del B_r$:
\begin{align}
	\int_{\del B_r} \abs{\Tr \phi_u \phi_v}^2
	&\leq C_{HS}\, c_0^2 t^{n-k-1} \kappa_u^2(t) \kappa_v^2(t) a^{1/16}\ .
	\label{eq:dichotomy-main-proof-bound-on-integral-of-second-casimir}
\end{align}

\paragraph{Part 4}
Our next goal is to similarly determine an $L^2$-bound for the function $\abs{\phi_u} \cdot \abs{\phi_v}$ on $\del B_r$, where $\abs{\phi_u}$ denotes pointwise application of the norm induced by the Killing form.
It is a property of $\mathfrak{su}(2)$ that
\begin{align}
	\abs{[\phi_u,\phi_v]}^2 = 4 \abs{\phi_u}^2 \abs{\phi_v}^2 - 4 \Tr\left( \phi_u \phi_v \right)^2\ .
\end{align}
Moreover, since $u$ and $v$ are orthonormal $\norm{[\phi\wedge\phi_v]}^2 \geq \abs{[\phi_u,\phi_v]}^2 \mu_{W^n}$, so we can bound the following integral with the help of \eqref{eq:dichotomy-main-proof-N-boundary-integral2}.
\begin{align}
	\int_{\del B_r} \abs{\phi_u}^2 \abs{\phi_v}^2 - \int_{\del B_r} \abs{ \Tr(\phi_u \phi_v) }^2
	= \frac{1}{4} \int_{\del B_r} \abs{[\phi_u, \phi_v]}^2
	\leq \frac{1}{4} \int_{\del B_r} \nu_v
	\leq 2\: t^{n-k-3}\: \kappa_v^2(t)\: a^{1/16}
\end{align}
Together with the main result of Part 3 in equation \eqref{eq:dichotomy-main-proof-bound-on-integral-of-second-casimir} this leads to
\begin{align}
	\int_{\del B_r} \abs{\phi_u}^2 \abs{\phi_v}^2
	\leq \left( C_{HS} + \frac{2}{c_0^2 t^2 \kappa_u^2(t)} \right) c_0^2 t^{n-k-1}\; \kappa_u^2(t)\; \kappa_v^2(t)\; a^{1/16}\ .
\end{align}
Observe that things have been set up in such a way that $t^2 \kappa_u^2(t)>1$:
First, $\kappa_u$ is almost non-decreasing, so $\kappa_u^2(t) \geq (t/R)^{2\epsilon} \kappa_u^2(R)$.
Second, we know that $\kappa^2_u(R) \geq \lambda(R)$ since the latter is the smallest eigenvalue of $T(R)$.
Third, we have previously chosen $r_1$ large enough that $\lambda(R) > (R/r_1)^{2a} \kappa^2(R)$.
Fourth, $t$ is larger than $r_1^{1-\sqrt{100}a}$ while $R>1$.
Plugging everything together yields
\begin{align}
	t^2 \kappa_u^2(t) \geq  r_1^{2(1-100\sqrt{a}-4a)} \kappa^2(r_1) > r_1^{2a} \kappa^2(r_1)\ ,
\end{align}
where the last step follows via $a^{1/4n} < 0.1$ and $n\geq 2$.
Finally, recall that we have also explicitly chosen $r_1$ to be large enough that the rightmost expression is larger than $1$.

The upshot of this part is that there exists a constant $C>1$ that only depends on $n$ and $k$, such that
\begin{align}
	\int_{\del B_r} \abs{\phi_u}^2 \abs{\phi_v}^2
	\leq C t^{n-k-1}\; \kappa_u^2(t)\; \kappa_v^2(t)\; a^{1/16}
	\label{eq:dichotomy-main-proof-bound-on-integral-of-first-casimir-product}\ .
\end{align}
From now on we allow the value of $C$ to increase from one equation to the next.

\paragraph{Part 5}
In this final part, we combine the results of the previous four parts with item \textit{(ii)} of \autoref{lem:dichotomy-a-priori-bounds}.
Thus, write $\Omega_1$ for the subset of $\del B_r$ where ${\norm{\phi_u} \leq \frac{\kappa_u(t)}{2 \vol X}}$.

The inequality in \eqref{eq:dichotomy-main-proof-bound-on-integral-of-first-casimir-product} remains true if we restrict the domain of integration to $\del B_r \setminus {\Omega_1}$, such that
\begin{align}
	\frac{\kappa_u^2(t)}{4\vol X} \int_{\del B_r \setminus {\Omega_1}} \abs{\phi_v}^2
	\leq \int_{\del B_r \setminus {\Omega_1}} \abs{\phi_u}^2 \abs{\phi_v}^2
	\leq C^2 t^{n-k-1}\; \kappa_u^2(t)\; \kappa_v^2(t)\; a^{1/16}\ .
\end{align}
More concisely, this leads to the inequality
\begin{align}
	\int_{\del B_r \setminus {\Omega_1}} \abs{\phi_v}^2 \leq C \vol X t^{n-k-1} \kappa_v^2(t) a^{1/16}\ .
\end{align}
Meanwhile, item $(i)$ of \autoref{lem:dichotomy-a-priori-bounds} provides the upper bound $\abs{\phi_v}\leq c_0^2 \kappa_v^2(t)$.
Writing $\vol \Omega_1 = f_{\Omega_1} \vol B_r$ this leads to
\begin{align}
	\int_{{\Omega_1}} \abs{\phi_v}^2 \leq f_{\Omega_1} (1+\epsilon) \vol X\: r^{n-k-1} c_0^2 \kappa_v^2(t)\ .
\end{align}
Item $(ii)$ of \autoref{lem:dichotomy-a-priori-bounds} states that $f_{\Omega_1} \leq c_1 \epsilon + c_2 \sqrt{N_u(t)}$.
Recall from Part 1 that $N_u(t) \leq a^{1/8}$ while $\epsilon = a/2$, such that $f_{\Omega_1}$ is bounded by some multiple of $a^{1/16}$.

It follows that the sum of the two integrals satisfies
\begin{align}
{\int_{\del B_r} \abs{\phi_v}^2 \leq C t^{n-k-1} \kappa_v^2(t) a^{1/16}}\ .
\end{align}
This is equivalent to the statement that ${\kappa_v^2(r) \leq C a^{1/16} \kappa_v^2(t)}$.
Finally, combining this with the bound $\kappa_v^2(r) \geq \frac{1}{4} \kappa_v^2(t)$ from Part 1 culminates in the inequality
\begin{align}
	\kappa_v^2(t) \leq 4C a^{1/16} \kappa_v^2(t)\ ,
\end{align}
which is absurd, as we are free to choose $a$ arbitrarily small and in particular such that $a^{1/16} < \frac{1}{4C}$.

\end{proof}


\section{Proof of Taubes' Dichotomy for Kapustin-Witten Solutions}
\label{sec:dichotomy-proof2}

In this section we prove \autoref{bigthm:dichotomy2}, which enhances \autoref{thm:dichotomy} for solutions of the Kapustin-Witten equations on four-manifolds.
We again closely follow Taubes' arguments, who proved an analogous statement in case the four-manifold is Euclidean space $\R^4$.
\begin{theorem}
\label{thm:dichotomy2}
Let $W^4$ be an ALX gravitational instanton of dimension $4$, with asymptotic fibers of dimension $k \leq 3$, and such that sectional curvature is bounded from below.
Assume $(A,\phi)$ are solutions of the $\theta$-Kapustin-Witten equations and if $\theta\not\equiv 0,\pi$ also assume that $\int_{W^4} \norm{F_A}^2 < \infty$, then
\begin{enumerate}
	\item either there is an $a>0$ such that $\liminf_{r\to\infty} \frac{\kappa(r)}{r^a} > 0$,
	\item or $[\phi \wedge \phi]=0$, $\nabla^A \phi = 0$, and $A$ is self-dual if $\theta = 0$, flat if $\theta \in (0,\pi)$, and anti-self-dual if $\theta = \pi$.
\end{enumerate}
\end{theorem}

\begin{proof}
Since solutions of the Kapustin-Witten equations also satisfy equation \eqref{eq:dichotomy-main-assumption}, the dichotomy of \autoref{thm:dichotomy} holds.
It remains to show that in the case where $[\phi\wedge\phi]$ is identically zero, also $\nabla^A \phi$ vanishes and $A$ is either (anti-)self-dual or flat as stated.
Hence, assume from now on that $[\phi\wedge\phi]= 0$.
At points where $\phi$ is non-zero the Higgs field can then be written as $\phi = \omega \otimes \sigma$, where $\omega\in\Omega^1(M)$ and $\sigma \in \Gamma(M,\ad E)$, normalized such that $\norm{\sigma} = 1$.

Consider first the case where $\theta=0$ (this also covers the case $\theta = \pi$ by a reversal of orientation).
The Kapustin-Witten equations then state on the one hand $F_A^+ = 0$, so $A$ is anti-self-dual, and on the other hand $(d_A \phi)^-= 0$ and $d_A \hodge \phi = 0$.
The latter two equations can only be satisfied if the constituents of $\phi$ satisfy $\nabla^A \sigma = 0$ and $(d \omega)^- = 0 = d \hodge \omega$.
In particular $\sigma$ is guaranteed to be covariantly constant at points where $\phi\neq 0$.

The zero locus of $\phi$ coincides with the zero locus of $\omega$.
Since the one-form $\omega$ satisfies the first order differential equations from above, it is an example of what Taubes calls a $\Z /2$ harmonic spinor in \cite{Taubes2014}.
In that article he investigates the zero locus of such $\Z /2$ harmonic spinors in general and Theorem 1.3 of \cite{Taubes2014} states that the zero locus has Hausdorff dimension 2.
The relevance for us is that the complement of the zeroes of $\omega$ in any given ball in $W^4$ is path connected.
This means that $\sigma$ can be defined at points where $\omega=0$ by parallel transport along paths where $\omega$ is non-zero.
Since $A$ is smooth and $\sigma$ is $\nabla^A$-parallel, parallel transport along two different paths results in the same value.

Since $\sigma$ is defined everywhere, the same is true for $\omega = \frac{1}{2}\Tr(\phi \sigma)$.
The elliptic differential equations $(d \omega)^- = 0 = d\hodge \omega$ imply that $\norm{\omega}^2$ is subharmonic.
Thus, by a classical result of Yau \cite[Theorem 3 \& Appendix (ii)]{Yau1976}, either $\norm{\omega}^2$ is constant or $\lim_{r\to\infty} r^{-1} \int_{B_r} \norm{\omega}^2 > 0$.
The latter is precluded by our assumptions and we conclude that $\nabla^A \phi =  0$.

In the case that $\theta \not\equiv 0 \pmod{\pi}$, neither of the terms $(d_A \phi)^\pm$ vanishes automatically.
However, if we additionally assume that $F_A$ is square-integrable, we can employ Uhlenbeck's compactness theorem for (anti-)self-dual connections to deduce that $A$ must be flat.
To see this, first note that we can find a coefficient $c(\theta)$ such that the connection $A + c(\theta) \phi$ satisfies the $\pi/2$ version of the Kapustin-Witten equations, so it is sufficient to consider the case $\theta=\pi/2$.

With that in mind and since $[\phi\wedge\phi]=0$, the Kapustin-Witten equations state that $F_A = \hodge d_A \phi$.
As a consequence the connection $\hat A := A+\phi$ is self-dual.
Since $\int_{W^4} \lVert F_{\hat A} \rVert^2 < \infty$, this connection $\hat A$ is the pullback of a smooth, regular connection on the one-point compactification of $W^4$, by Uhlenbeck's compactness theorem.
It follows that at large radius the field strength falls off faster than the volume of geodesic balls grows, i.e. $\lVert F_{\hat A} \rVert \leq \frac{c}{r^{4-k}}$.
Keeping this in mind, note that $\nabla^{\hat A} \phi = \nabla^A \phi$, such that $F_{\hat A} = 2 (d_{\hat A} \phi)^+$ and
\begin{align}
	\int_{B_r} \lVert F_{\hat A}\rVert^2 = \int_{B_r} \Tr F_{\hat A} \wedge d_{\hat A} \phi = \int_{\del B_r} \Tr F_{\hat A} \wedge \phi\ ,
\end{align}
where we have used Stokes' theorem and the Bianchi identity in the last equality.
With the bound on $\lVert F_{\hat A}\rVert$ the integral on the right is bounded by a multiple of $\frac{\kappa}{r}$, which approaches $0$ for $r\to \infty$, so $\hat A$ is flat.
From here we are back in the situation where $(d_A \phi)^+ = 0 = d_A \hodge \phi$ and the same argument as before leads to $\nabla^A \phi = 0$.

\end{proof}

\printbibliography[heading=bibintoc]

\end{document}